\documentclass[11pt]{amsart}

\usepackage{amsfonts,amsthm,amsmath,enumerate,amssymb,latexsym,color,tikz,mathrsfs,bm,bbm}
\usepackage[text={6in,9in},centering]{geometry}
\usepackage[backref=page,
            colorlinks,
            linkcolor=blue,
            anchorcolor=red,
            citecolor=red]{hyperref}
\usetikzlibrary{intersections}
\usepackage{tikz-cd}
\usepackage{multicol}

\theoremstyle{plain}
\newtheorem{theorem}{Theorem}[section]
\newtheorem{lemma}[theorem]{Lemma}
\newtheorem{corollary}[theorem]{Corollary}
\newtheorem{proposition}[theorem]{Proposition}

\newtheorem{question}[theorem]{Question}

\theoremstyle{definition}
\newtheorem{definition}[theorem]{Definition}
\newtheorem{example}[theorem]{Example}

\theoremstyle{remark}
\newtheorem{remark}[theorem]{Remark}

\numberwithin{equation}{section}

\DeclareMathOperator{\dimh}{dim_H}
\DeclareMathOperator{\dist}{dist}
\DeclareMathOperator{\diam}{diam}
\DeclareMathOperator{\dims}{dim_S}
\DeclareMathOperator{\dimb}{dim_B}
\DeclareMathOperator{\lip}{Lip}
\DeclareMathOperator{\dimub}{\overline{\dim}_B}

\newcommand{\rd}{\,\mathrm{d}}  

\def\R{\mathbb R}

\def\i{\bm i}
\def\j{\bm j}

\def\t{\bm t}
\def\bi{\mathbf i}
\def\bj{\mathbf j}

\def\h{\mathcal H}

\begin{document}

\title[Coding Multiplicity in Self-Similar Sets on the Line]{Coding Multiplicity in Self-Similar Sets on the Line}

\author{K\'aroly Simon}
\address{Department of Stochastics, Institute of Mathematics, Budapest University of Technology and Economics, Budapest, Hungary \& HUN-REN-BME Stochastics Research Group, Budapest, Hungary}
\email{karoly.simon51@gmail.com}

\author{Jian-Ci Xiao}
\address{School of Mathematics, Nanjing University of Aeronautics and Astronautics, Nanjing, China}
\email{jcxiao@nuaa.edu.cn}


\begin{abstract}
    We study coding multiplicity for self-similar sets on the real line. When the similarity dimension exceeds one, almost every point with respect to the natural self-similar measure has infinitely many codings. We focus on the case when the similarity dimension is at most one. Our analysis proceeds  through the recurrence properties and the size of the critical set, defined as the union of the pairwise intersections of the first-level pieces.

    Our main structural result shows that when the critical set is finite, unboundedness of coding multiplicity is equivalent to periodic recurrence to the critical set. Moreover, every commensurable system with a finite critical set has uniformly bounded coding multiplicity. This conclusion is sharp in the incommensurable setting. We construct a one-parameter family of four-map systems such that for Lebesgue almost every parameter, the critical set is a singleton and the open set condition fails, while the attractor contains a point with infinitely many codings. More generally, for transversal families, the set of parameters admitting an infinite coding fibre is dense among those for which the strong separation condition fails. We also obtain a quantitative lower bound for the Hausdorff dimension of this parameter set. Such parameters can form a meagre set.

    We continue the study of the geometric size of branching, obtaining dimension formulae for typical coding fibres and almost-sure upper bounds on the dimension of the critical set in transversal families. In contrast, we construct an exponentially separated system with no dimension drop whose critical set has the same Hausdorff dimension as the attractor and whose coding map has a fibre of continuum cardinality. Although we work in the one-dimensional setting, several results and arguments extend to higher dimensions.
\end{abstract}

\maketitle

\section{Introduction}
\label{sec:introduction}
\subsection{Background and motivation}

Let $\Phi=\{\varphi_i\}_{i\in\Lambda}$ be an iterated function system (IFS) of contracting similitudes on $\R$, where $\Lambda$ is a finite alphabet with $\#\Lambda\geq 2$. The attractor of $\Phi$ is the unique non-empty compact set $K$ satisfying $K=\bigcup_{i\in\Lambda} \varphi_i(K)$; see Hutchinson~\cite{Hut81}. We assume throughout that $K$ is not a singleton.  Let $I$ be the convex hull of $K$. Then $\varphi_i(I)\subset I$ for all $i\in\Lambda$, and $K$ can be obtained by the iterative construction 
\[
    K = \bigcap_{n=1}^\infty \bigcup_{i_1\cdots i_n\in\Lambda^n} \varphi_{i_1}\circ\cdots\circ\varphi_{i_n}(I).
\]
We call each $\varphi_{i_1}\circ\cdots\circ\varphi_{i_n}(K)$ a \emph{cylinder} of $K$ and $\varphi_{i_1}\circ\cdots\circ\varphi_{i_n}(I)$ the corresponding \emph{cylinder interval}.

On the symbolic side, the full shift $\Sigma:=\Lambda^{\mathbb{N}}$ equipped with the left shift $\sigma$ provides a natural combinatorial model. The \emph{coding map} $\Pi \colon \Sigma \to K$ is defined by
\[
    \Pi(i_1i_2\cdots) := \lim_{n\to\infty} \varphi_{i_1}\circ\cdots\circ\varphi_{i_n}(0).
\]
For a point $x\in K$, the fibre $\Pi^{-1}(x)$ is the set of symbolic codings (or addresses) of $x$, and its cardinality records the \emph{coding multiplicity} of $x$. We are interested in the possible cardinalities of these fibres and in how overlaps generate large coding multiplicity.

At one extreme, if $\Phi$ satisfies the \emph{strong separation condition} (SSC), the first-level pieces $\varphi_i(K)$ are pairwise disjoint and every point in $K$ has a unique coding. The weaker \emph{open set condition} (OSC) allows intersections between the pieces but still imposes a uniform bound on coding multiplicity. More precisely, there is a constant $M<\infty$ such that every $x\in K$ belongs to at most $M$ cylinder intervals of comparable size; see Schief~\cite{Sch94} and Bandt--Graf~\cite{BG92}. We isolate the corresponding bound for cylinders by saying that $\Phi$ satisfies the \emph{uniformly bounded coding condition} (UBC) if there exists $M<\infty$ such that 
\[
    \#\{i_1\cdots i_n\in\Lambda^n: x\in \varphi_{i_1}\circ\cdots\circ\varphi_{i_n}(K)\} \leq M, \quad \forall x\in K, \forall n\geq 1.
\]
Equivalently, $\Phi$ satisfies the UBC if and only if $\sup_{x\in K}\#\Pi^{-1}(x) < \infty$. Thus the OSC implies the UBC.

To describe where codings branch, consider the \emph{critical set}
\[
    \mathcal{C} := \bigcup_{i\neq j} \varphi_i(K)\cap\varphi_j(K).
\] 
A classical question asks how small this first level overlap can be when the OSC fails. For connected self-similar sets in the plane, Bandt and Rao~\cite{BR07} proved that the finiteness of $\mathcal{C}$ already implies the OSC. Without connectedness, the situation is quite different. Kamalutdinov and Tetenov~\cite{KT20} constructed a totally disconnected self-similar set on $\R$ for which $\mathcal{C}$ consists of a single point yet the OSC fails. Moreover, every point of the attractor has at most two codings and hence the UBC survives. Kadirova~\cite{Kad26} recently obtained an orientation-preserving analogue of this construction. 

The critical set determines the set of points with more than one coding. Indeed, if $\Lambda^*$ denotes the set of finite words including the empty word, then 
\begin{equation}\label{eq:branchingorbit}
	\{x\in K: \#\Pi^{-1}(x) \geq 2\} = \mathcal{C}\cup\bigcup_{n\geq 1}\bigcup_{i_1,\cdots,i_n\in\Lambda} \varphi_{i_1}\circ\cdots\circ\varphi_{i_n}(\mathcal{C}).
\end{equation}
To see this, take two distinct codings of the same point. Removing their longest common prefix produces two codings of a point in $\mathcal{C}$ with different initial symbols. Conversely, every image of a point in $\mathcal{C}$ has at least two codings. Formula~\eqref{eq:branchingorbit} reveals two ways in which the critical set influences coding multiplicity. The \emph{size} of $\mathcal{C}$ controls the global extent of non-uniqueness. For instance, 
if $\mathcal{C}$ is finite, the set of points with more than one coding is countable. The \emph{recurrence} of the shift orbit to $\Pi^{-1}(\mathcal{C})$, on the other hand, controls how non-uniqueness accumulates along an individual fibre. 

Our results address these two aspects of overlap. We first characterise unbounded coding multiplicity when the critical set is finite. We then study the dimensions of coding fibres and the dimension and measure of the critical set. Finally, we investigate these questions in transversal families, estimating both the parameters admitting infinite fibres and the critical sets of typical systems. Although we focus on the one-dimensional setting for clarity of exposition, the symbolic, recurrence, and measure-theoretic arguments developed below are formulated independently of dimension.

\begin{remark}
    The distinction between the size and recurrence of $\mathcal{C}$ has already appeared in the theory of $\beta$-expansions. For $1<\beta<2$, the IFS $\{\frac{x}{\beta},\frac{x+1}{\beta}\}$ identifies codings with expansions $x=\sum_{n=1}^\infty\varepsilon_n\beta^{-n}$, where $\varepsilon_n\in\{0,1\}$. Sidorov~\cite{Sid03} proved that Lebesgue-almost every point in the attractor interval has continuum many $\beta$-expansions. On the other hand, the existence and structure of points having a unique expansion, finitely many expansions, or countably many expansions depend on the parameter $\beta$; see, among others,~\cite{GS01,BS14,Bak15}. This illustrates that a large first level overlap may make non-uniqueness abundant, while the precise multiplicity of exceptional fibres is determined by the much finer recurrence and branching structure. Similar phenomena occur in more general overlapping IFSs~\cite{DJKL19,DJKLX21}.
\end{remark}


\subsection{Finite critical sets and recurrent branching}

We say that $\Phi$ satisfies the \emph{finite intersection property} (FIP) if $\mathcal{C}$ is finite. The FIP differs from the classical notion of post-critical finiteness. 
In Kigami's terminology~\cite{Kig01}, a self-similar structure is \emph{post-critically finite} (p.c.f.) when $\bigcup_{n\geq1}\sigma^n(\Pi^{-1}(\mathcal{C}))$ is finite. 
Postcritical finiteness implies the FIP, but the converse fails. Singh and Tetenov~\cite{ST18}, for example, constructed a planar self-similar dendrite satisfying the one-point intersection property whose postcritical set is countably infinite.

Deng and Lau~\cite{DL08} proved that a p.c.f.\ self-similar IFS with commensurable similarity matrices satisfies the OSC. We ask what remains true when the p.c.f. hypothesis is weakened to the FIP. As shown in Lemma~\ref{lem:syminfiniteslice}, a point has infinitely many codings precisely when the shift orbit of one of its codings meets $\Pi^{-1}(\mathcal{C})$ infinitely often. The simplest instance is periodic recurrence. If for some non-empty finite word $\i$, the fixed point of $\varphi_{\i}$ belongs to $\mathcal{C}$, then the periodic coding $\i^\infty$ returns to $\Pi^{-1}(\mathcal{C})$ infinitely often, and that fixed point admits infinitely many codings. 

Under the FIP, such periodic recurrence is essentially the only mechanism by which coding multiplicity can become unbounded. We call $\Phi$ \emph{commensurable} if $\{r_i:i\in\Lambda\}\subset r^{\mathbb Q}$ for some $0<r<1$, where $r_i\in(0,1)$ denotes the contraction ratio of $\varphi_i$.

\begin{theorem}\label{thm:fipimplybd}
    Assume that $\Phi$ satisfies the FIP. Then the following are equivalent:
    \begin{enumerate}
        \item The IFS $\Phi$ fails the UBC;
        \item there is a point $x\in K$ with infinitely many codings;
        \item there is a non-empty finite word $\i$ such that the fixed point of $\varphi_{\i}$ belongs to $\mathcal{C}$.
    \end{enumerate}
    Moreover, if $\Phi$ is commensurable then condition (2) cannot occur; so $\Phi$ satisfies the UBC.
\end{theorem}


Viewing $K$ as the horizontal direction and $\Sigma$ as the vertical direction in the coding relation, $\mathcal{C}$ describes first level ``horizontal'' overlap, while the fibres describe ``vertical'' multiplicity. In general, finiteness of the critical set need not bound the cardinalities of the fibres. For commensurable systems, however, Theorem~\ref{thm:fipimplybd} shows that finiteness of $\mathcal{C}$ does yield such a uniform bound. So the UBC is the separation condition that survives when the p.c.f.\ hypothesis in Deng and Lau~\cite{DL08} is weakened to the FIP.

For incommensurable systems, Theorem~\ref{thm:fipimplybd} sheds new light on known counterexamples to the implication FIP $\Longrightarrow$ OSC. In the Kamalutdinov--Tetenov example, and in the orientation-preserving variant of Kadirova~\cite{Kad26}, the failure of the OSC is not due to recurrent branching since every point has at most two codings. Instead, it is incommensurability which allows a sequence of distinct neighbour maps to converge to the identity, and the OSC therefore fails by the Bandt--Graf criterion~\cite{BG92}. Our next construction shows that the opposite coding behaviour can occur with exactly a singleton critical set.

\begin{theorem}\label{thm:fipnewway}
    There exists a one-parameter family of self-similar IFSs with four maps on $\R$ such that, for Lebesgue almost every parameter,
    \begin{enumerate}
        \item the critical set consists of a single point;
        \item the IFS fails the OSC; and
        \item the associated attractor contains a point with infinitely many codings.
    \end{enumerate}
    In particular, the UBC fails.
\end{theorem}

Thus the cardinality of the critical set alone does not determine coding multiplicity. Our construction is also more economical: the earlier examples use six maps, whereas ours uses only four. Its proof combines the recurrence criterion in Theorem~\ref{thm:fipimplybd} with a transversality argument developed in~\cite[Section 7.3]{BSS23}.

\subsection{Coding fibres and critical sets: dimension and measure}

The preceding results concern fibre cardinalities when the critical set is geometrically small. We now turn to fibre dimensions and the dimension and measure of the critical set.

Let $s=s(\Phi)$ be the \emph{similarity dimension} of $\Phi$, namely the unique solution to $\sum_{i\in\Lambda} r_i^s=1$, and let $\nu=\sum_{i\in\Lambda} r_i^s\varphi_i\nu$ be the \emph{natural self-similar measure}. By Schief's criterion~\cite{Sch94}, if the OSC fails then $\mathcal H^s(K)=0$. On the other hand, equipped with its natural metric, the symbolic space $\Sigma$ has positive $s$-dimensional Hausdorff measure. It is then natural to ask how the loss of $\mathcal{H}^s$-measure or of dimension under $\Pi$ is reflected in the size of the fibres. The following observation makes this quantitative at the level of dimension.

\begin{proposition}\label{prop:fibrebound}
    Let $\nu$ be the natural self-similar measure. 
    \begin{enumerate}
        \item For $\nu$-almost every $x\in K$, 
        \begin{equation}\label{eq:fibredim}
            \dimh\Pi^{-1}(x) = \dimb\Pi^{-1}(x) = s-\dimh\nu.
        \end{equation}
        Consequently, if $\dimh\nu<s$, then $\Pi^{-1}(x)$ has positive Hausdorff dimension and hence continuum cardinality for $\nu$-almost every $x\in K$. Also, 
        \begin{equation}\label{eq:nddequa}
            \dimh\nu = s \,\,\Longleftrightarrow\,\, \dimb\Pi^{-1}(x)=0 \text{ for $\nu$-almost every $x\in K$}.
        \end{equation}
        \item If $\Phi$ satisfies the exponential separation condition (ESC; see Definition~\ref{def:esc}), then 
        \[
            \dim_{\textup{A}}\Pi^{-1}(x)\leq\max\{s-1,0\} \quad \text{ for every } x\in K.
        \]
    \end{enumerate}
\end{proposition}

Part (1) follows directly from the dimension conservation principle of Feng and Hu~\cite{FH09} and an elementary local covering estimate. In particular, from~\eqref{eq:nddequa} we see that the UBC implies $\dimh\nu=s$. Part (2) uses Shmerkin's $L^q$-dimension estimate~\cite{Shm19}. Thus a strict drop in dimension forces a strong symbolic manifestation. This regime is closely related to the \emph{exact overlap conjecture} originally proposed by the first author~\cite{Sim96}, which predicts that a drop below the expected dimension $\min\{1,s\}$ can occur only if two distinct finite words define the same cylinder map. An exact overlap immediately generates large fibres because repeated substitutions between the two words produce a point with continuum many codings. Major progress on this conjecture has been made by Hochman, Rapaport, Varj\'u and others; see~\cite{Hoc14,Var19,Rap22,RV24} and the references therein.  

When $\dim_H\nu = s$, the situation becomes subtler. Proposition~\ref{prop:fibrebound} gives only zero-dimensional typical fibres. Together with an earlier joint work of the first author and Kolossv\'ary~\cite{KS19}, we get the following equivalences:
\begin{align*}
    \dimh\nu = s \,\,&\Longleftrightarrow\,\, \dimb\Pi^{-1}(x)=0 \text{ for $\nu$-almost every $x\in K$} \\
    &\Longleftrightarrow\,\, \exists \text{ Borel } B\subset\Sigma \text{ with } \mu(B)=0 \text{ and $\Pi$ is injective on $\Sigma\setminus B$} \\
    &\Longleftrightarrow\,\, \mu_{\xi(\bi)}=\delta_{\bi} \text{ for $\mu$-almost every $\bi\in\Sigma$},
\end{align*}
where $\mu$ is the natural Bernoulli measure on $\Sigma$ with weights $(r_i^s)$, $\xi(\bi)=\Pi^{-1}(\Pi(\bi))$ and $\mu_{\xi(\bi)}$ is the conditional measure on the fibre $\xi(\bi)$ (see~\cite[Chapter 12]{BSS23}). However, these equivalences do not by themselves determine the cardinalities of the full coding fibres. Additional codings, if present, may lie in the discarded $\mu$-null set.

By~\eqref{eq:branchingorbit}, if $\dimh\mathcal{C}<\dimh\nu$ then
\[
	\nu\bigl(\{x\in K:\#\Pi^{-1}(x)\geq2\} \bigr)=0.
\]
Under the OSC such a strict bound is known; see~\cite{Mor99, Ols02}. Without the OSC, the critical set and exceptional fibres can be much larger. The following example shows that even exponential separation, together with the absence of dimension drop, does not prevent the critical set from having the full Hausdorff dimension of the attractor, nor does it prevent exceptional fibres from having large cardinality. 

\begin{proposition}\label{exa:escwithnddcri}
    Set $u:=\sum_{k=0}^\infty 4^{-2^k}$ and consider
    \[
        \Phi = \Big\{ \varphi_0(x) = \frac{x}{4}, \varphi_u(x)=\frac{x}{4}+u, \varphi_1(x)=\frac{x}{4}+1 \Big\}.
    \]
    Let $K$ be its attractor. Then $\Phi$ satisfies the ESC and
    \[
        \dimh (\varphi_0(K)\cap\varphi_u(K)) = \dimh K = \frac{\log 3}{\log 4}.
    \]
    Moreover, there exists a point $x_*\in K$ with $\#\Pi^{-1}(x_*)=2^{\aleph_0}$.
\end{proposition}

The example also naturally raises the question of whether $\mathcal{C}$ can nevertheless carry positive mass with respect to the natural self-similar measure. Bandt and Graf~\cite{BG92} asserted that $\dim_H K = s$ implies $\nu(\mathcal{C})=0$, but a gap in their argument was later pointed out in~\cite{SS07}. The conclusion does hold when $\mathcal{P}^s(K)>0$, because then $\mathcal{P}^s|_K$ is a constant multiple of $\nu$ and $\mathcal{P}^s(\mathcal{C})$ always vanishes; see~\cite{PSS00}. The remaining case is thus $\mathcal{P}^s(K)=0$. In that case, the $s$-density of $\nu$ is infinite $\nu$-almost everywhere. We obtain partial information as follows.

\begin{proposition}\label{thm:singularpieces}
    Let $\mathbf{p}=(p_i)_{i\in\Lambda}$ be a probability vector with $p_i>0$ and let $\nu_{\mathbf{p}}=\sum_{i\in\Lambda} p_i\varphi_i\nu_{\mathbf{p}}$ be the associated self-similar measure. 
    \begin{enumerate}
        \item If $\dimh\nu_{\mathbf{p}}=\dims\nu_{\mathbf{p}}$, then $\varphi_i\nu_{\mathbf{p}}\perp\varphi_j\nu_{\mathbf{p}}$ for all distinct $i, j\in\Lambda$;
        \item For the natural self-similar measure $\nu$, if $\nu(\mathcal{C})>0$ then for $\nu$-almost every $x\in K$, there is $c_x>0$ such that
        \[
            \nu(B(x,r)) \geq c_x r^s\log(1/r) \quad \text{ for all } 0<r<1.
        \]
    \end{enumerate}
\end{proposition}

Here $\dims\nu_{\mathbf{p}}$ denotes the \emph{similarity dimension} of the measure $\nu_{\mathbf{p}}$; see~\eqref{eq:simdimofselfsimmea}. The first assertion states that, whenever a self-similar measure has dimension equal to its similarity dimension, its first level component measures are mutually singular. The second assertion gives a quantitative consequence of positive critical mass. After completing the manuscript, we learned from De-Jun Feng that he and Meng Wu~\cite{FW26} had independently obtained the same observation as Proposition~\ref{thm:singularpieces}(1). Combining this with a projection theorem in~\cite{Wu25}, they proved that the attractor of a homogeneous IFS on the line consisting of uncountably many similitudes has Hausdorff dimension one.

\subsection{Branching in transversal families}

We finally study coding fibres and critical sets in transversal families. Consider a family
\[
    \Phi^{\t} = \{\varphi_i^{\t}(x)=\lambda_i(\t)x+d_i(\t)\}_{i\in\Lambda}, \quad \t\in U,
\]
where $U\subset\R^d$ is a bounded open set.

Transversality methods in this context date back to Pollicott and Simon~\cite{PS95} and have since been developed in several directions; see, for example, ~\cite{PS00, SSU01, Ram05} as well as the textbook~\cite{BSS23}. Roughly speaking, the transversality condition requires that coincidences among distinct symbolic branches be non-degenerate in the parameter; see Definition~\ref{de:trans} for a precise formulation. 

The following result gives a density statement and a dimension estimate for parameters admitting an infinite coding fibre.

\begin{theorem}\label{thm:denseinffibre}
    Let $\{\Phi^{\t}\}_{\t\in U}$ be a $C^1$ family of self-similar IFSs on $\mathbb{R}$ with attractor $K^{\t}$. Denote by $\widetilde{U}$ the set of parameters $\t\in U$ for which $\Phi^{\t}$ fails the SSC. If the family satisfies the transversality condition, then the set 
    \begin{equation}\label{eq:defofT}
        T := \{\t\in U: \exists x\in K^{\t} \text{ with infinitely many codings} \}
    \end{equation}
    is dense in $\widetilde{U}$. Moreover, for every $\t_0\in \widetilde{U}$ and every open set $V\subset U$ containing $\t_0$, 
    \[
        \dimh (T\cap V) \geq d-1+\dimh K^{\t_0}.
    \]
    Consequently, 
    \begin{equation}\label{eq:dimofdense2}
        \dimh T \geq d-1+\sup_{\t\in\widetilde{U}} \dimh K^{\t}.
    \end{equation}
\end{theorem}

The density conclusion cannot in general be strengthened to residuality. Indeed, by adding a transversal translation perturbation to the Kamalutdinov--Tetenov example, we obtain a family on which $T$ is meagre in $\widetilde{U}$; see Section~\ref{subsec:baireexa} for details.

\begin{remark}
    Given a fixed self-similar IFS $\Phi=\{\varphi_i\}_{i=1}^m$, the associated \emph{translation family} is $\Phi^{\t}=\{x\mapsto\varphi_i(x)+t_i\}_{i=1}^m$, where $\t=(t_1,\ldots,t_m)$ ranges over a non-degenerate open ball $U\subset \mathbb{R}^{m}$. These systems are natural \emph{translational perturbations} of $\Phi$. Let
    \[
        \rho(\Phi) := \max\limits_{i\neq j} (r_i+r_j),
    \] 
    where $r_i$ denotes the contraction ratio of $\varphi_i$. 
    It was proved in~\cite[Lemma 2.14]{BKRS21} that $\{\Phi^{\t}\}_{\t\in U}$ satisfies the transversality condition if $\rho(\Phi)<1$. So if $0\in U$, $\rho(\Phi)<1$ and the unperturbed system $\Phi=\Phi^0$ fails the SSC, Theorem~\ref{thm:denseinffibre} produces arbitrarily small translational perturbations whose attractors contain a point with infinitely many codings.
\end{remark}

The second result concerns the size of the critical set
\[
    \mathcal{C}^{\t}:=\bigcup_{i\neq j} \varphi^{\t}_i(K^{\t}) \cap \varphi^{\t}_j(K^{\t}).
\]

\begin{theorem}\label{thm:transfam_drop}
    Let $\{\Phi^{\t}\}_{\t\in U}$ be a $C^1$ family of self-similar IFSs on $\R$ with attractor $K^{\t}$. If the family satisfies the transversality condition, then for Lebesgue almost every $\t\in U$, 
    \[
        \dimub \mathcal{C}^{\t} \leq \max\{0, 2s(\t)-1\},
    \]
    where $s(\t)$ is the similarity dimension. 
\end{theorem}

In particular, if $\sup_{\t\in U} s(\t)<1$ then the upper bound in Theorem~\ref{thm:transfam_drop} is strictly smaller than $s(\t)$. Hence at almost every parameter at which no dimension drop occurs, the natural self-similar measure assigns full mass to the set of uniquely coded points. 

There are related generic sparsity results for transversal systems. Rams~\cite{Ram05} proved that, for typical parameters, intersections of cylinders are nearly minimal at many scales. His estimates yield, among other consequences, positivity of packing measure in the subcritical regime, a result previously obtained by Peres, Solomyak and the first author~\cite{PSS00}. Theorem~\ref{thm:transfam_drop} addresses a different aspect of the overlap: it directly bounds the dimension of the critical set $\mathcal{C}^{\t}$. 



\subsection{Organization of the paper}

Section~\ref{sec:codingmul} develops the recurrence theory of coding for systems with finite critical sets and proves Theorems~\ref{thm:fipimplybd} and~\ref{thm:fipnewway}. Section~\ref{sec:dimoffibresandc} studies the dimensional and measure structure of coding fibres and critical sets, including Proposition~\ref{prop:fibrebound}, Proposition~\ref{exa:escwithnddcri} and Proposition~\ref{thm:singularpieces}. Section~\ref{sec:transversalfam} proves the parametric results in Theorems~\ref{thm:denseinffibre} and~\ref{thm:transfam_drop}. The final section presents several open questions.

\smallskip
\noindent{\bf Notation}. Throughout the paper we write $A \lesssim B$ to mean $A \le C B$ for a constant $C>0$ independent of the quantities under consideration, and $A \lesssim_V B$ if the implied constant may depend on the parameter(s) $V$. If both $A \lesssim B$ and $B \lesssim A$ hold we write $A \approx B$, and similarly $A \approx_V B$. We denote by $\dimh$ the Hausdorff dimension, by $\dimub$ the upper box-counting dimension, by $\dim_{\textup{A}}$ the Assouad dimension, and by $\mathcal H^s$ the $s$-dimensional Hausdorff measure. The open ball of radius $r$ centred at $x$ is $B(x,r)$. Unless otherwise stated, the logarithms are to base 2.

\section{Coding multiplicity and recurrence}
\label{sec:codingmul}

Throughout Sections~\ref{sec:codingmul} and~\ref{sec:dimoffibresandc}, we fix a self-similar IFS $\Phi=\{\varphi_i\}_{i\in\Lambda}$ on $\R$ with attractor $K$ and similarity dimension $s$. Write 
\[
    \varphi_i(x) = \lambda_i x+a_i, \quad 0<r_i:=|\lambda_i|<1.
\]
So $r_i$ is the contraction ratio of $\varphi_i$. Without loss of generality, assume that $K\subset B(0,1)$.

\subsection{Symbolic preliminaries}
Let $\Lambda^*:=\bigcup_{n=0}^\infty\Lambda^n$ be the collection of finite words over $\Lambda$, where $\Lambda^0:=\{\emptyset\}$. 
For $\i=i_1\cdots i_p\in\Lambda^*$, write 
\[
    \lambda_{\i}:=\lambda_{i_1}\cdots\lambda_{i_p}, \, r_{\i}:=|\lambda_{\i}|, \, \varphi_{\i} := \varphi_{i_1}\circ\cdots\circ\varphi_{i_p}, \, K_{\i}:=\varphi_{\i}(K),
\]
with the conventions $\varphi_{\emptyset}=\textup{id}$ and $\lambda_\emptyset=1$. We denote by $[\i]$ the cylinder set consisting of all infinite words with prefix $\i$. The symbolic space is $\Sigma:=\Lambda^{\mathbb{N}}$ and the coding map $\Pi\colon\Sigma\to K$ is defined by 
\[
    \Pi(\bi) = \lim_{n\to\infty} \varphi_{i_1}\circ\cdots\circ\varphi_{i_n}(0), \quad \bi=i_1i_2\cdots\in\Sigma.
\]
Equip $\Sigma$ with the standard ultrametric 
\[
    \rho(\omega,\tau) = \left\{\begin{array}{ll} r_{\omega\wedge\tau}, & \text{if } \omega\neq\tau, \\ 0, & \text{if } \omega=\tau, \end{array}\right.
\]
where $\omega\wedge\tau$ is the longest common prefix of $\omega$ and $\tau$. Unless otherwise stated, all dimensions of subsets of $\Sigma$ below are taken with respect to $\rho$. Let $\sigma$ be the left shift and let $\mathcal{M}_\sigma(\Sigma)$ denote the space of $\sigma$-invariant Borel probability measures on $\Sigma$. 

For $0<\delta<1$, define the stopping set
\begin{equation}\label{eq:lambdan}
    \Lambda_\delta := \{\i\in\Lambda^*\setminus\{\emptyset\}: r_{\i}\leq \delta <r_{\i^-}\},
\end{equation}
where $\i^-$ is obtained by deleting the last symbol of $\i$. Letting $\underline{r}:=\min_{i\in\Lambda} r_i$, we have for every $\i\in\Lambda_\delta$ the tight bounds $\underline{r}\delta<r_{\i}\leq\delta$. Moreover, every infinite word has a unique prefix in $\Lambda_\delta$. We also write $\bi|_n$ for the prefix of $\bi\in\Sigma$ of length $n$.

Given a probability vector $\mathbf{p}=(p_i)_{i\in\Lambda}$ with $p_i>0$, let $\mu_{\mathbf{p}}$ be the Bernoulli measure on $\Sigma$ with weights $p_i$, and let $\nu_{\mathbf{p}}:=\Pi\mu_{\mathbf{p}}$ be the associated self-similar measure. Its \emph{similarity dimension} is 
\begin{equation}\label{eq:simdimofselfsimmea}
    \dims \nu_{\mathbf{p}} := \frac{\sum_{i\in\Lambda} p_i\log p_i}{\sum_{i\in\Lambda} p_i\log r_i}.
\end{equation}
For the choice $p_i=r_i^s$, we simply write $\mu=\mu_{\mathbf{p}}$ and $\nu=\nu_{\mathbf{p}}$. So $\nu$ is the \emph{natural} self-similar measure. Note that $\mu([\i])=r_{\i}^s$ for every $\i\in\Lambda^*$ and $\dims\nu=s$.

\subsection{Recurrence and infinite fibres}

Before proving Theorem~\ref{thm:fipimplybd}, we give a symbolic criterion for the existence of an infinite fibre. Let $\mathcal{C}:=\bigcup_{i\neq j}K_i\cap K_j$ and let $\widetilde{C}:=\Pi^{-1}(\mathcal{C})$. Since $\mathcal{C}$ is closed in $K$ and $\Pi$ is continuous, $\widetilde{C}$ is closed in $(\Sigma,\rho)$. 

\begin{lemma}\label{lem:syminfiniteslice}
    There exists a point $x\in K$ such that $\#\Pi^{-1}(x)=\infty$ if and only if there exist $\omega\in\Sigma$ and integers $n_k\nearrow\infty$ such that $\sigma^{n_k}(\omega)\in \widetilde{C}$ for all $k\geq 1$.
\end{lemma}
\begin{proof}
    Assume that $\#\Pi^{-1}(x)=\infty$. By compactness, the fibre has an accumulation point $\omega$. We may choose distinct $\bi_k=i_{k,1}i_{k,2}\cdots\in \Pi^{-1}(x)\setminus\{\omega\}$ such that the length of $\omega\wedge\bi_k$, denoted by $n_k$, goes to infinity. Passing to a subsequence if necessary, we may assume that $n_k$ is strictly increasing. Clearly,
    \begin{align*}
        \varphi_{\omega\wedge\bi_k}\bigl(\Pi(\sigma^{n_k}\omega)\bigr) = \Pi(\omega)= x = \Pi(\bi_k) = \varphi_{\omega\wedge\bi_k}\bigl(\Pi(\sigma^{n_k}\bi_k)\bigr),
    \end{align*}
    and hence \(\Pi(\sigma^{n_k}\omega)=\Pi(\sigma^{n_k}\bi_k)\). Since $\sigma^{n_k}\omega$ and $\sigma^{n_k}\bi_k$ begin with distinct symbols but have the same projection, $\sigma^{n_k}(\omega)\in \widetilde{C}$.

    Conversely, if such $\omega\in\Sigma$ and $\{n_k\}_k$ exist, then for each $k$ there exists $\bj_k=j_{k,1}j_{k,2}\cdots\in\Sigma$ with $\Pi(\bj_k)=\Pi(\sigma^{n_k}\omega)$ and $j_{k,1}\neq \omega_{n_k+1}$. Setting $x:=\Pi(\omega)$, we have
    \[
        \Pi(\omega|_{n_k}\bj_k) = \varphi_{\omega|_{n_k}}(\Pi(\sigma^{n_k}\omega)) = \Pi(\omega) = x.
    \]
    So $\omega|_{n_k}\bj_k\in\Pi^{-1}(x)$. Moreover, $\omega|_{n_k}\bj_k$ differs from $\omega$ at position $n_k+1$. The sequences $\omega|_{n_k}\bj_k$ are pairwise distinct: if $k<\ell$, then $\omega|_{n_\ell}\bj_\ell$ agrees with $\omega$ through position $n_\ell$, and hence at position $n_k+1$, whereas $\omega|_{n_k}\bj_k$ does not. Therefore, $\Pi^{-1}(x)$ contains infinitely many distinct words.
\end{proof}

Thus an infinite fibre exists precisely when the shift orbit of some coding visits the symbolic critical set $\widetilde{C}$ infinitely often. Poincar\'e recurrence then gives a useful measure-theoretic sufficient condition.

\begin{lemma}[Poincar\'e recurrence]
    If there exists $m\in\mathcal{M}_\sigma(\Sigma)$ such that $m(\widetilde{C})>0$, then $m$-almost every word of $\widetilde{C}$ returns to $\widetilde{C}$ infinitely often under $\sigma$.
\end{lemma}

A simple way to obtain such a measure is when $\widetilde{C}$ contains a non-empty closed subset $A$ with $\sigma^p(A)\subset A$ for some $p\geq 1$. In this case, $\sigma^p|_A$ is a continuous self-mapping of the compact space $(A,\rho)$. By the Krylov-Bogolyubov theorem, $\sigma^p$ admits an invariant Borel probability measure $m$ supported on $A$. Then $\tilde{m}:=\frac{1}{p}\sum_{k=0}^{p-1} \sigma^k m$ is $\sigma$-invariant with $\tilde{m}(A) \geq m(A)/p>0$.

\begin{example}\label{exa:periodictype}
    Suppose that, for some non-empty word $\i\in\Lambda^*$, the fixed point $x_{\i}$ of $\varphi_{\i}$ belongs to $\mathcal{C}$. Then $\i^\infty:=\i\i\cdots\in \widetilde{C}$ returns to $\widetilde{C}$ periodically. Equivalently, choose $\bj\neq\i^\infty$ with $\Pi(\bj)=x_{\i}$. Since $\varphi_{\i}(x_{\i})=x_{\i}$,
    \[
        x_{\i} = \Pi(\bj) = \Pi(\i\bj) = \Pi(\i\i\bj) = \cdots,
    \]
    and these codings are pairwise distinct. Thus $x_{\i}$ has infinitely many codings.
\end{example}

Under the FIP, periodic recurrence turns out to be the \emph{only} way to produce unbounded coding multiplicity. To see this, it is convenient to record the following elementary lemmas.

\begin{lemma}\label{lem:tree}
    Let $T$ be a rooted tree in which every vertex has at least one child and at most $M\geq 1$ children. If every infinite path in $T$ has at most $N$ branching vertices (that is, vertices which have at least two children), then $T$ has at most $M^N$ infinite paths.
\end{lemma}
\begin{proof}
    We argue by induction on $N$. If $N=0$, no vertex is branching. So $T$ has exactly one infinite path.
    For $N\geq 1$, if $T$ has no branching vertex, then it again has exactly one path. Otherwise, let $v$ be a branching vertex closest to the root. The part of the tree above $v$ is a single chain, so every infinite path of $T$ passes through one of the (at most $M$) children of $v$. In each of the subtrees rooted at those children, no path can have more than $N-1$ additional branching vertices. By the induction hypothesis, each such subtree has at most $M^{N-1}$ infinite paths. Thus $T$ has at most $M\cdot M^{N-1}=M^N$ infinite paths. 
\end{proof}

\begin{lemma}\label{lem:Eisenough}
    Assume that $\mathcal{C}\neq\varnothing$. Then
    \begin{enumerate}
        \item $\sup_{x\in K} \#\Pi^{-1}(x) = \sup_{x\in \mathcal{C}} \#\Pi^{-1}(x)$;
        \item If $\i,\i'\in\Lambda^*$ are distinct and neither is a prefix of the other, $\#(K_{\i}\cap K_{\i'})\leq \# \mathcal{C}$.
    \end{enumerate} 
\end{lemma}
\begin{proof}
    (1) If $x\notin\mathcal{C}$ then all codings of $x$ share the same first digit. Removing this common digit preserves the cardinality of the fibre. If a fibre has more than one element, this procedure must eventually reach $\mathcal{C}$. So the equality holds.  (2) Since neither $\i$ nor $\i'$ is a prefix of the other, $K_{\i} \cap K_{\i'} \subset \varphi_{\i\wedge\i'}(\mathcal{C})$ and the conclusion follows immediately.
\end{proof}

\begin{proof}[Proof of Theorem~\ref{thm:fipimplybd} (the equivalence)]
    Let $N:=\#\mathcal{C}<\infty$. By Lemma~\ref{lem:Eisenough}, we immediately see that (1) $\Longrightarrow$ (2). The implication (3) $\Longrightarrow$ (1) follows from Example~\ref{exa:periodictype}, so it suffices to check (2) $\Longrightarrow$ (3).
    
    Choose $x\in K$ with $\#\Pi^{-1}(x)=\infty$. Consider the rooted tree $T_x:=\{\i\in\Lambda^*: x\in K_{\i}\}$, with an edge from $\i$ to $\i j$ whenever $j\in\Lambda$ and $\i j\in T_x$. Every vertex of $T_x$ has at least one child, because $x\in K_{\i}$ implies $x\in K_{\i j}$ for some $j\in\Lambda$, and at most $\#\Lambda$ children. Since each coding of $x$ corresponds to an infinite path in $T_x$, there are more than $(\#\Lambda)^N$ infinite paths. 
    
    By Lemma~\ref{lem:tree}, some coding $\omega=\omega_1\omega_2\cdots\in\Pi^{-1}(x)$ must contain at least $N+1$ branching prefixes. Let $0\leq n_1<n_2<\cdots<n_{N+1}$ be the lengths of the first $N+1$ such prefixes. For $1\leq k\leq N+1$, define $x_k:=\varphi_{\omega|_{n_k}}^{-1}(x)$. Since $\omega|_{n_k}$ is branching in $T_x$, there exist distinct $i,j\in\Lambda$ such that 
    \[
        x\in \varphi_{\omega|_{n_k}}(K_i\cap K_j).
    \]
    So $x_k\in K_i\cap K_j \subset\mathcal{C}$ for each $k$. Since $\mathcal{C}$ has only $N$ elements, by the pigeonhole principle there exists $p<q$ such that $x_p=x_q$. Let $\i:=\omega_{n_p+1}\cdots\omega_{n_q}$. Then 
    \[
        x = \varphi_{\omega|_{n_p}}(x_p) = \varphi_{\omega|_{n_q}}(x_q) = \varphi_{\omega|_{n_p}}(\varphi_{\i}(x_q)).
    \]
    By injectivity of $\varphi_{\omega|_{n_p}}$, $x_p=\varphi_{\i}(x_q)=\varphi_{\i}(x_p)$. So $x_p$ is the fixed point of $\varphi_{\i}$ and belongs to $\mathcal{C}$.
\end{proof}

The next subsection provides an inhomogeneous example satisfying the third condition in Theorem~\ref{thm:fipimplybd}. However, as we will prove shortly, this phenomenon cannot occur for commensurable IFSs with the FIP and hence is likewise excluded in the homogeneous case.

\begin{lemma}\label{lem:atleastp}
    If $\Phi$ is commensurable, then for any $p\geq 2$, we can find $q\geq 1$ such that
    \[
        \#\{\omega\in\Lambda^*\setminus\{\i\}: \lambda_{\omega}=\lambda_{\i}\} \geq p, \quad \forall \i\in\Lambda^q.
    \]
\end{lemma}
\begin{proof}
    Since the ratios are rational powers of a common base, for every $i\in\Lambda$ we can find $j_i\in\Lambda\setminus\{i\}$ and positive integers $a_i,b_i$ such that $\lambda_i^{a_i}=\lambda_{j_i}^{b_i}$. Write $A:=\max_{i\in\Lambda} a_i$ and take $q>pA\cdot\#\Lambda$. For any $\i=i_1\cdots i_q\in\Lambda^q$, the pigeonhole principle implies that some digit, say $i_*$, appears at least $pA\geq pa_{i_*}$ times. For $k=1,\ldots,p$, let $\i^{(k)}$ be obtained from $\i$ by deleting the first $ka_{i_*}$ occurrences of $i_*$ and then prefixing $kb_{i_*}$ copies of $j_{i_*}$. Then 
    \[
        \lambda_{\i^{(k)}} = \lambda_{j_{i_*}}^{kb_{i_*}}\cdot\lambda_{\i}\lambda_{i_*}^{-ka_{i_*}} = \lambda_{\i}.
    \]
    The words $\i^{(1)},\ldots,\i^{(p)}$ are pairwise distinct and none equals $\i$, because their numbers of occurrences of the symbol $i_*$ are pairwise distinct. This proves the lemma.
\end{proof}

\begin{corollary}
    Suppose $\Phi$ is commensurable and satisfies the FIP. Let $\i,\i'\in\Lambda^*$ be distinct with $\lambda_{\i}=\lambda_{\i'}$ and let $p\geq 2$. Let $q$ be given by Lemma~\ref{lem:atleastp} with respect to $p$. If there is $\j\in\Lambda^q$ such that $K_{\i\j}\cap K_{\i'\j}\neq\varnothing$, then $\#(K_{\i}\cap K_{\i'})\geq p$.
\end{corollary}
\begin{proof}
    Since $K_{\i\j} \cap K_{\i'\j} \neq\varnothing$, there are $x,x'\in K$ with $\varphi_{\i\j}(x)=\varphi_{\i'\j}(x')$. For any $\omega\in\Lambda^*$ with $\lambda_{\omega}=\lambda_{\j}$, we have by $\lambda_{\i}=\lambda_{\i'}$ that
    \begin{align*}
        \varphi_{\i\omega}(x) - \varphi_{\i'\omega}(x') &= \big( \lambda_{\i\omega}x+\varphi_{\i\omega}(0) \big) - \big( \lambda_{\i'\omega}x' + \varphi_{\i'\omega}(0) \big) \\
        &= \lambda_{\i\omega}(x-x') + \varphi_{\i\omega}(0) - \varphi_{\i'\omega}(0) \\
        &= \lambda_{\i\omega}(x-x') + \varphi_{\i\j}(0) - \varphi_{\i'\j}(0) \\
        &= \varphi_{\i\j}(x) - \varphi_{\i'\j}(x') \\
        &= 0.
    \end{align*}
    In particular, $\varphi_{\i\omega}(x)\in K_{\i\omega}\cap K_{\i'\omega}$ for every such $\omega$. By Lemma~\ref{lem:atleastp}, there are distinct words $\omega_1,\ldots,\omega_p\in\Lambda^*\setminus\{\j\}$ such that $\lambda_{\omega_k}=\lambda_{\j}$ for $1\leq k\leq p$. The points $\varphi_{\i\omega_1}(x),\ldots,\varphi_{\i\omega_p}(x)$ are pairwise distinct. Otherwise, if $\varphi_{\i\omega_k}(x)=\varphi_{\i\omega_\ell}(x)$ for some $k\neq \ell$, then $\varphi_{\omega_k}(x)=\varphi_{\omega_{\ell}}(x)$ and then $\varphi_{\omega_k}=\varphi_{\omega_{\ell}}$ since $\lambda_{\omega_k}=\lambda_{\omega_\ell}$; but such an exact overlap contradicts the FIP. As a consequence, $\#(K_{\i}\cap K_{\i'})\geq p$.
\end{proof}

\begin{proof}[Proof of Theorem~\ref{thm:fipimplybd} (the commensurable case)]
    Let $r\in(0,1)$ be a common base for the contraction ratios. By commensurability, there is a constant $C>0$ such that the set $\{\lambda_{\i}r^{-n}: \i\in\Lambda_{r^n}\}$ has at most $C$ elements for every $n$. Assume for contradiction that the UBC fails. By the already proved equivalence, there is a point $x_*\in \mathcal{C}$ with infinitely many codings. Let $p:=\#\mathcal{C}+1$ and choose $q$ as in Lemma~\ref{lem:atleastp} for this $p$. Also set $M:=\#\Lambda$ and choose $N>C(M+1)^{q}$.
    
    Choose $N$ distinct codings of $x_*$. Once $n$ is sufficiently large, their prefixes in $\Lambda_{r^n}$ are pairwise distinct. So for all large $n$, there are at least $N$ distinct words $\i_1,\ldots,\i_N\in\Lambda_{r^n}$ such that $x_*\in\bigcap_{t=1}^N K_{\i_t}$. We also take $n$ so large that every $\i_t$ has length at least $q$. By the pigeonhole principle, among these $N$ words there are $(M+1)^q$ with the same value of $\lambda_{\i_t}$. For each $1\leq t\leq N$, write $\i_t=\omega_t\j_t$, where $\j_t\in\Lambda^{q}$ is its tail of length $q$. Since $\#\Lambda^{q}=M^{q}<(M+1)^q$, the pigeonhole principle again yields two indices $t_1<t_2$ with $\j_{t_1}=\j_{t_2}$ and $\lambda_{\omega_{t_1}\j_{t_1}}=\lambda_{\i_{t_1}}=\lambda_{\i_{t_2}}=\lambda_{\omega_{t_2}\j_{t_2}}$. Since $\i_{t_1}\neq \i_{t_2}$, we have $\omega_{t_1}\neq\omega_{t_2}$ while $\lambda_{\omega_{t_1}}=\lambda_{\omega_{t_2}}$. Recall that 
    \[
        x_*\in K_{\i_{t_1}}\cap K_{\i_{t_2}} = K_{\omega_{t_1}\j_{t_1}}\cap K_{\omega_{t_2}\j_{t_2}}.
    \]
    By the previous corollary, $\#(K_{\omega_{t_1}}\cap K_{\omega_{t_2}}) \geq p = \#\mathcal{C}+1$. Since $\omega_{t_1}\neq\omega_{t_2}$ but $\lambda_{\omega_{t_1}}=\lambda_{\omega_{t_2}}$, neither word is a prefix of the other. Lemma~\ref{lem:Eisenough}(2) then gives $\#(K_{\omega_{t_1}}\cap K_{\omega_{t_2}})\leq \#\mathcal{C}$, a contradiction.
\end{proof}

\subsection{A singleton critical set with an infinite fibre}

For comparison, below we recall the Kamalutdinov--Tetenov example. 

\begin{example}[\cite{KT20}]\label{exa:KT20}
    There are real values $p,q,r\in(0,1)$ such that the IFS $\{\varphi_i\}_{i=1}^6$, where 
    \begin{equation*}
        \begin{gathered}
            \varphi_1(x) = px, \, \varphi_2(x)=rx+\tfrac{1}{5}, \, \varphi_3(x) = -qx+\tfrac{8}{15}, \\
            \varphi_4(x) = rx+\tfrac{8}{15}-r, \, \varphi_5(x) = -rx+\tfrac{4}{5}, \, \varphi_6(x) = rx+1-r,
        \end{gathered}
    \end{equation*}
    does not satisfy the OSC, whereas the only intersection is $\varphi_3(K)\cap\varphi_4(K)=\{\varphi_3(0)=\varphi_4(1)\}$.
    Also, the Hausdorff dimension of the attractor equals the similarity dimension of $\Phi$. Please see Figure~\ref{fig:KT20} for an illustration. Since $0$ and $1$ both have a unique coding, every point in this self-similar set has at most two codings (recall Lemma~\ref{lem:Eisenough}). 

    \begin{figure}[htbp]
        \centering
        \begin{tikzpicture}[scale=1, >=stealth]
            \draw[thick,->] (0,0) to (1,0);
            \node at (0,0)[draw, fill, rectangle, minimum width=1pt, minimum height=3pt, inner sep=0pt]{};
            \node at (1,0)[draw, fill, rectangle, minimum width=1pt, minimum height=3pt, inner sep=0pt]{};
            \node at(0.5,-0.3) {$\varphi_1$};
            \draw[thick,->] (2,0) to (2.8,0);
            \node at (2,0)[draw, fill, rectangle, minimum width=1pt, minimum height=3pt, inner sep=0pt]{};
            \node at (2.8,0)[draw, fill, rectangle, minimum width=1pt, minimum height=3pt, inner sep=0pt]{};
            \node at(2.4,-0.3) {$\varphi_2$};
            \draw[thick,->] (16/3-0.8,0) to (16/3,0);
            \node at (16/3-0.8,0)[draw, fill, rectangle, minimum width=1pt, minimum height=3pt, inner sep=0pt]{};
            \node at (16/3,0)[draw, fill, rectangle, minimum width=1pt, minimum height=3pt, inner sep=0pt]{};
            \node at(16/3-0.4,-0.3) {$\varphi_4$};
            \draw[thick,<-] (16/3-0.6,0.2) to (16/3,0.2);
            \node at (16/3-0.6,0.2)[draw, fill, rectangle, minimum width=1pt, minimum height=3pt, inner sep=0pt]{};
            \node at (16/3,0.2)[draw, fill, rectangle, minimum width=1pt, minimum height=3pt, inner sep=0pt]{};
            \node at(16/3-0.3,0.5) {$\varphi_3$};
            \draw[thick,<-] (7.2,0) to (8,0);
            \node at (7.2,0)[draw, fill, rectangle, minimum width=1pt, minimum height=3pt, inner sep=0pt]{};
            \node at (8,0)[draw, fill, rectangle, minimum width=1pt, minimum height=3pt, inner sep=0pt]{};
            \node at(7.6,-0.3) {$\varphi_5$};
            \draw[thick,->] (9.2,0) to (10,0);
            \node at (9.2,0)[draw, fill, rectangle, minimum width=1pt, minimum height=3pt, inner sep=0pt]{};
            \node at (10,0)[draw, fill, rectangle, minimum width=1pt, minimum height=3pt, inner sep=0pt]{};
            \node at(9.6,-0.3) {$\varphi_6$};
            \draw[thick,dashed] (0,0) to (0,-0.8);
            \draw[thick,dashed] (2,0) to (2,-0.8);
            \draw[thick,dashed] (16/3,0) to (16/3,-0.8);
            \draw[thick,dashed] (8,0) to (8,-0.8);
            \draw[thick,dashed] (10,0) to (10,-0.8);
            \node at (0,-1.1) {$0$};
            \node at (2,-1.1) {$1/5$};
            \node at (16/3,-1.1) {$8/15$};
            \node at (8,-1.1) {$4/5$};
            \node at (10,-1.1) {$1$};
            \draw[thick] (0,-0.8) to (10,-0.8);
            \node at (0,-0.8)[draw, fill, rectangle, minimum width=1pt, minimum height=3pt, inner sep=0pt]{};
            \node at (10,-0.8)[draw, fill, rectangle, minimum width=1pt, minimum height=3pt, inner sep=0pt]{};
        \end{tikzpicture}
        \caption{The Kamalutdinov--Tetenov example}
        \label{fig:KT20}
    \end{figure}
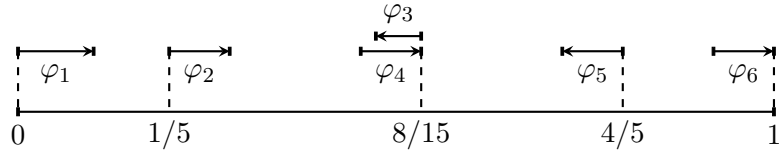
\end{example}

Very recently, Kadirova~\cite{Kad26} constructed an orientation-preserving variant. The system in her example also satisfies the UBC. We now construct a family for which the critical set is still a singleton but the critical point has infinitely many codings. This gives a different obstruction to the OSC.

\smallskip
\paragraph{{\bf Construction for Theorem~\ref{thm:fipnewway}}}

Fix $0<\varepsilon<1/100$ and $0<p<\varepsilon/2$. For $q\in[\varepsilon/2,\varepsilon]$, define the IFS $\Phi_q := \{\varphi_{q,0},\varphi_{q,1},\varphi_{q,2},\varphi_{q,3}\}$, where 
\[
    \varphi_{q,0}(x)=px,\, \varphi_{q,1}(x)=qx+\frac{1-q}{2},\, \varphi_{q,2}(x)=px+\frac{1}{2}-p,\, \varphi_{q,3}(x)=px+1-p.
\]
Let $K_q$ be the attractor and $\Pi_q: \Sigma\to K_q$ the coding map, where $\Sigma:=\{0,1,2,3\}^{\mathbb{N}}$. Here we keep the parameter $q$ as a subscript to prevent confusion with iteration.

For an illustration, please see Figure~\ref{fig:newfipinf}. Note that the convex hull of every $K_q$ is the unit interval $[0,1]$; the only possible first level overlap is between $\varphi_{q,1}(K_q)$ and $\varphi_{q,2}(K_q)$, and
\[
    1/2 = \varphi_{q,1}(1/2) = \varphi_{q,2}(1), \quad \forall q\in[\varepsilon/2,\varepsilon].
\]
Hence $1/2\in\varphi_{q,1}(K_q)\cap\varphi_{q,2}(K_q)$ for all $q$.

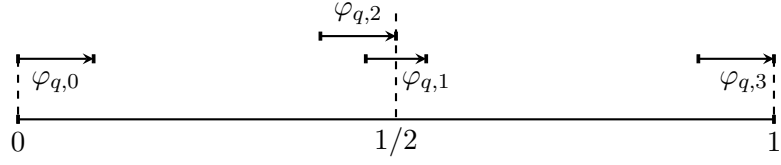
\begin{figure}[htbp]
    \centering
    \begin{tikzpicture}[scale=1, >=stealth]
        \draw[thick,->] (0,0) to (1,0);
        \node at (0,0)[draw, fill, rectangle, minimum width=1pt, minimum height=3pt, inner sep=0pt]{};
        \node at (1,0)[draw, fill, rectangle, minimum width=1pt, minimum height=3pt, inner sep=0pt]{};
        \node at(0.5,-0.3) {$\varphi_{q,0}$};
        \draw[thick,->] (4.6,0) to (5.4,0);
        \node at (4.6,0)[draw, fill, rectangle, minimum width=1pt, minimum height=3pt, inner sep=0pt]{};
        \node at (5.4,0)[draw, fill, rectangle, minimum width=1pt, minimum height=3pt, inner sep=0pt]{};
        \node at(5.4,-0.3) {$\varphi_{q,1}$};
        \draw[thick,->] (4,0.3) to (5,0.3);
        \node at (4,0.3)[draw, fill, rectangle, minimum width=1pt, minimum height=3pt, inner sep=0pt]{};
        \node at(4.5,0.6) {$\varphi_{q,2}$};
        \draw[thick,->] (9,0) to (10,0);
        \node at (5,0.3)[draw, fill, rectangle, minimum width=1pt, minimum height=3pt, inner sep=0pt]{};
        \node at (9,0)[draw, fill, rectangle, minimum width=1pt, minimum height=3pt, inner sep=0pt]{};
        \node at (10,0)[draw, fill, rectangle, minimum width=1pt, minimum height=3pt, inner sep=0pt]{};
        \node at(9.6,-0.3) {$\varphi_{q,3}$};
        \draw[thick,dashed] (0,0) to (0,-0.8);
        \draw[thick,dashed] (5,0.6) to (5,-0.8);
        \draw[thick,dashed] (10,0) to (10,-0.8);
        \node at (0,-1.1) {$0$};
        \node at (5,-1.1) {$1/2$};
        \node at (10,-1.1) {$1$};
        \draw[thick] (0,-0.8) to (10,-0.8);
        \node at (0,-0.8)[draw, fill, rectangle, minimum width=1pt, minimum height=3pt, inner sep=0pt]{};
        \node at (10,-0.8)[draw, fill, rectangle, minimum width=1pt, minimum height=3pt, inner sep=0pt]{};
    \end{tikzpicture}
    \caption{The IFS $\Phi_q$ with one-point intersection}
    \label{fig:newfipinf}
\end{figure}

\begin{theorem}
    For almost all $q\in[\varepsilon/2,\varepsilon]$, we have
    \[
        \bigcup_{i\neq j} \varphi_{q,i}(K_q)\cap\varphi_{q,j}(K_q) = \{1/2\}.
    \]
    In particular, the IFS has the FIP and fails the OSC, while the middle point $1/2$ possesses infinitely many codings.
\end{theorem}

The last claim follows because $1/2$ already has the two distinct codings $1^\infty$ and $23^\infty$; Example~\ref{exa:periodictype} then yields infinitely many codings. To prove the remaining part, we record some derivative estimates. For any finite word $\i=i_1\cdots i_n$, write $\varphi_{q,\i}:=\varphi_{q,i_1}\circ\cdots\circ\varphi_{q,i_n}$ for the composition. Set 
\[
    F_n(q) := \max_{\i\in\{0,1,2,3\}^n} \Big| \frac{\partial}{\partial q}\varphi_{q,\i}(0) \Big|.
\]
From $\varphi_{q,i\i}=\varphi_{q,i}\circ\varphi_{q,\i}$, we obtain 
\[
    \frac{\partial}{\partial q}\varphi_{q,i\i}(0) = \left\{\begin{array}{ll} p\cdot\frac{\partial}{\partial q}\varphi_{q,\i}(0), & i=0,2,3, \\ \varphi_{q,\i}(0)+q\cdot\frac{\partial}{\partial q}\varphi_{q,\i}(0)-\frac{1}{2}, & i=1. \end{array}\right.
\]
Since $\varphi_{q,\i}(0)\in[0,1]$, $|\varphi_{q,\i}(0)-\frac{1}{2}|\leq\frac{1}{2}$. Hence (recall that $p<q$)
\[
    F_n(q) \leq \max\Big\{ pF_{n-1}(q), qF_{n-1}(q)+\frac{1}{2} \Big\} \leq qF_{n-1}(q)+\frac{1}{2}.
\]
An induction easily yields 
\[
    F_n(q) \leq \frac{1}{2}\sum_{k=0}^{n-1} q^k \leq \frac{1}{2(1-q)}.
\]

Note that for every $\omega\in\Sigma$, the finite composition $\varphi_{q,\omega|_n}(0)$ converges uniformly to $\Pi_q(\omega)$, where $\omega|_n$ is the $n$-prefix of $\omega$, and their derivatives also converge uniformly on this interval; see Lemma~\ref{lem:uniformc1}. Writing $D_q(\omega) := \frac{\partial}{\partial q}\Pi_q(\omega)$, we have 
\[
    |D_q(\omega)| = \Big| \lim_{n\to\infty} \frac{\partial}{\partial q} \varphi_{q,\omega|_n}(0) \Big| \leq \frac{1}{2(1-q)},
\]
So 
\begin{equation}\label{eq:fip_supderi}
    M := \sup_{q\in[\varepsilon/2,\varepsilon]}\sup_{\omega\in\Sigma} |D_q(\omega)| \leq \frac{1}{2(1-\varepsilon)}. 
\end{equation}
For $k\in\{0,1,2\}$, we have the simple estimate 
\[
    \Pi_q(k\omega) = \varphi_{q,k}(\Pi_q(\omega)) \leq \max_{x\in[0,1]} \varphi_{q,k}(x) \leq \max\Big\{ p, \frac{1+q}{2},\frac{1}{2} \Big\} \leq \frac{1+\varepsilon}{2}, \quad \forall \omega\in\Sigma,
\]
and thus (recall that $\varepsilon<\frac{1}{100}$)
\begin{equation}\label{eq:dqkomega}
    |D_q(k\omega)| \leq M \leq \frac{1}{2(1-\varepsilon)} \leq 1-\varepsilon \leq 2(1-\Pi_q(k\omega)), \quad \forall \omega\in\Sigma.
\end{equation}

Assume that for some $q$ there exists $x\in\varphi_{q,1}(K_q)\cap\varphi_{q,2}(K_q)$ with $x\neq 1/2$. Choose $\bi\in[1]$ and $\bj\in[2]$ with $x=\Pi_q(\bi)=\Pi_q(\bj)$. Since $\Pi_q(1^\infty)=1/2$, $\bi\neq 1^\infty$ and thus $\bi=1^ni\bi'$ for some $n\geq 1$ and $i\neq 1$. If $i=3$, then 
\[
    x = \Pi_q(\bi) \geq \varphi_{q,1}^n\varphi_{q,3}(0) = \varphi_{q,1}^n(1-p) > \varphi_{q,1}^n(1/2) = 1/2,
\]
contradicting $x\in\varphi_{q,2}(K_q)\subset[1/2-p,1/2]$. If $i=2$, the tail after the $2$ cannot be all $3$'s: otherwise, $x=\varphi_{q,1}^n(\Pi_q(23^\infty))=\varphi_{q,1}^n(1/2)=1/2$, a contradiction. Consequently, either $\bi\in[1^n0]$ or $\bi\in[1^n23^\ell k]$ for some $\ell\geq 0$ and $k\in\{0,1,2\}$. Similarly, since $\Pi_q(23^\infty)=1/2$, the word $\bj$ cannot be $23^\infty$. Hence $\bj\in[23^mj]$ for some $m\geq 0$ and $j\in\{0,1,2\}$. 

Define two countable families of finite words
\[
    \mathcal{U} := \{1^n0: n\geq 1\} \cup \{1^n23^\ell k: n\geq 1, \ell\geq 0, k\in\{0,1,2\}\}
\]
and 
\[
    \mathcal{V} := \{23^mj: m\geq 0, j\in\{0,1,2\}\}.
\]
Then every extra intersection corresponds to some pair $(\bm{u},\bm{v})\in\mathcal{U}\times\mathcal{V}$.

The following lemma is a special case of~\cite[Lemma 7.3.7]{BSS23}.

\begin{lemma}\label{lem:bsstrans}
    Let $\bm{u}\in\mathcal{U}$ and $\bm{v}\in\mathcal{V}$. If there exists some $\delta>0$ such that for all $\bi\in[\bm{u}]$ and $\bj\in[\bm{v}]$, 
    \begin{equation}\label{eq:bsstrans}
        \forall q\in[\varepsilon/2,\varepsilon], |\Pi_q(\bi)-\Pi_q(\bj)|\leq\delta \Longrightarrow \Big| \frac{\partial}{\partial q}(\Pi_q(\bi)-\Pi_q(\bj)) \Big| \geq\delta,
    \end{equation}
    then 
    \[
        \dimh\{q\in[\varepsilon/2,\varepsilon]: \varphi_{q,\bm{u}}(K_q)\cap\varphi_{q,\bm{v}}(K_q)\neq\varnothing\} \leq 2\max_{q\in[\varepsilon/2,\varepsilon]}s(q),
    \]
    where $s(q)$ is the similarity dimension of $\Phi_q$.
\end{lemma}

For our parameters, $s(q)<1/2$; hence the right hand side is $<1$ and the corresponding set of $q$ has Lebesgue measure zero.

We now verify the transversality condition~\eqref{eq:bsstrans} for every pair $(\bm{u},\bm{v})\in\mathcal{U}\times\mathcal{V}$.  Let $\bi\in[\bm{u}]$ and $\bj\in[\bm{v}]$. 
By definition of $\mathcal{U}$, either $\bi=1^n0\omega$ or $\bi=1^n23^\ell k\omega$ for some $n\geq 1, \ell\geq 0, k\neq 3$ and $\omega\in\Sigma$. From $\varphi_{q,1}^n(x) = q^n(x-\frac{1}{2})+\frac{1}{2}$, we obtain $\Pi_q(\bi) = \frac{1}{2} - q^nA(q)$,
where 
\[
    A(q) := \left\{\begin{array}{ll} \frac{1}{2}-p\Pi_q(\omega), & \bi=1^n0\omega, \\ p^{\ell+1}(1-\Pi_q(k\omega)), & \bi=1^n23^\ell k\omega. \end{array}\right.
\]
Since $k\neq 3$ and $p<\varepsilon<1/100$, we have 
\[
    A(q) \geq \min\Big\{ \frac{1}{2}-p, p^{\ell+1}\frac{1-\varepsilon}{2} \Big\} =: c_{\bm{u}} >0,
\]
Moreover, for the derivative we have
\begin{itemize}
    \item if $\bi=1^n0\omega$ then by~\eqref{eq:fip_supderi},
    \[
        |A'(q)| = p|D_q(\omega)| \leq pM \leq \frac{p}{2(1-\varepsilon)} \leq \frac{1}{2}-\varepsilon \leq \frac{1}{2}-p\Pi_q(\omega) = A(q);
    \]
    \item if $\bi=1^n23^\ell k\omega$ then by~\eqref{eq:dqkomega},
    \[
        |A'(q)| = p^{\ell+1}|D_q(k\omega)| \leq 2p^{\ell+1}(1-\Pi_q(k\omega)) = 2A(q).
    \]
\end{itemize}
Thus in either case
\begin{equation}\label{eq:a'qle2aq}
    |A'(q)| \leq 2A(q), \quad \forall q\in[\varepsilon/2,\varepsilon].
\end{equation}
Similarly, by definition of $\mathcal{V}$, $\bj=23^mj\tau$ for some $m\geq 0$, $j\neq 3$ and $\tau\in\Sigma$. So $\Pi_q(\bj) = \frac{1}{2} - p^{m+1}B(q)$,
where $B(q):=1-\Pi_q(j\tau)$. 
Since $j\neq 3$, $B(q)>0$. By~\eqref{eq:dqkomega}, we also have
\begin{equation}\label{eq:b'qle2bq}
    |B'(q)| = |D_q(j\tau)| \leq 2(1-\Pi_q(j\tau)) = 2B(q), \quad \forall q\in[\varepsilon/2,\varepsilon].
\end{equation}

The difference of the two projections is 
\[
    \Pi_q(\bi) - \Pi_q(\bj) = \Big( \frac{1}{2}-q^nA(q) \Big) - \Big( \frac{1}{2}-p^{m+1}B(q) \Big) = -q^nA(q)+p^{m+1}B(q).
\]
Since $q\geq\varepsilon/2$ and $A(q)\geq c_{\bm{u}}$, 
\begin{equation}\label{eq:tildecu}
    q^nA(q) \geq \frac{\varepsilon^n}{2^n}c_{\bm{u}} =: \tilde{c}_{\bm{u}}. 
\end{equation}
Set $\delta:=\tilde{c}_{\bm{u}}/4$. 

Assume $|\Pi_q(\bi) - \Pi_q(\bj)|\leq\delta$. Then $p^{m+1}B(q) \leq q^nA(q) + \delta \leq \frac{5}{4}q^nA(q)$. Combining this with~\eqref{eq:a'qle2aq} and~\eqref{eq:b'qle2bq}, we have
\begin{align*}
    \Big| \frac{\partial}{\partial q}(\Pi_q(\bi)-\Pi_q(\bj)) \Big| &= |-nq^{n-1}A(q)-q^nA'(q)+p^{m+1}B'(q)| \\
    &\geq nq^{n-1}A(q)-q^n|A'(q)|-p^{m+1}|B'(q)| \\
    &\geq \frac{n}{q}\cdot q^nA(q) - 2q^nA(q) - 2p^{m+1}B(q) \\
    &\geq \Big( \frac{n}{q}-\frac{9}{2} \Big)q^nA(q).
\end{align*}
Since $n\geq 1$ and $q\leq\varepsilon<1/100$, $\frac{n}{q}-\frac{9}{2}\geq\frac{1}{2q}$. So 
\[
    \Big| \frac{\partial}{\partial q}(\Pi_q(\bi)-\Pi_q(\bj)) \Big| \geq \frac{1}{2q}q^nA(q) \geq \frac{1}{2q}\tilde{c}_{\bm{u}} > \frac{\tilde{c}_{\bm{u}}}{4} = \delta.
\]
Thus the transversality condition~\eqref{eq:bsstrans} is satisfied with $\delta=\tilde{c}_{\bm{u}}/4$.

As remarked after Lemma~\ref{lem:bsstrans}, the set
\[
    E_{\bm{u},\bm{v}}:=\{q\in[\varepsilon/2,\varepsilon]: \varphi_{q,\bm{u}}(K_q)\cap\varphi_{q,\bm{v}}(K_q)\neq\varnothing\}
\]
has Hausdorff dimension $<1$ and in particular has Lebesgue measure zero. Both $\mathcal{U}$ and $\mathcal{V}$ are countable, so the union $E:=\bigcup_{\bm{u}\in\mathcal{U}}\bigcup_{\bm{v}\in\mathcal{V}}E_{\bm{u},\bm{v}}$ is a null set. For every $q\in[\varepsilon/2,\varepsilon]\notin E$, there is no extra intersection in $\varphi_{q,1}(K_q)\cap\varphi_{q,2}(K_q)$ besides $1/2$; one easily checks that the other pairs of first level pieces are disjoint by construction. So  
\[
    \bigcup_{i\neq j} \varphi_{q,i}(K_q)\cap\varphi_{q,j}(K_q) = \{1/2\}, \quad \forall q\in [\varepsilon/2,\varepsilon]\setminus E,
\] 
which proves Theorem~\ref{thm:fipnewway}.

\section{Dimension of fibres and critical sets}
\label{sec:dimoffibresandc}


\subsection{Dimension of typical fibres}

We begin with an elementary estimate that will be used in both the typical and uniform parts of the argument. For $E\subset\Sigma$ and $\delta>0$, let $N_\delta(E)$ denote the least number of closed $\rho$-balls of radius $\delta$ required to cover $E$.

\begin{lemma}\label{lem:localfibreest}
    For every $x\in K$, 
    \[
        \dimub \Pi^{-1}(x) \leq s-\underline{d}_{\nu}(x), 
    \]
    where $\underline{d}_{\nu}(x) := \liminf_{r\to 0} \frac{\log\nu(B(x,r))}{\log r}$ is the lower local dimension of $\nu$ at $x$.
\end{lemma}
\begin{proof}
    For $0<\delta<1$ and $x\in K$, let
    \[
        \mathcal{I}_\delta(x)
        :=\{\i\in\Lambda_\delta:K_{\i}\cap B(x,\delta)\neq\varnothing\}.
    \]
    The pairwise disjoint cylinders $[\i]$, $\i\in \mathcal{I}_\delta(x)$, cover $\Pi^{-1}(x)$ and each has $\rho$-diameter at most $\delta$. Therefore, $N_\delta(\Pi^{-1}(x))\leq \#\mathcal{I}_\delta(x)$ and
    \begin{equation}\label{eq:fibre-cover-local-dimension}
        \dimub\Pi^{-1}(x) \leq \limsup_{\delta\to 0} \frac{\log \#\mathcal{I}_\delta(x)}{-\log\delta}.
    \end{equation}
    For every $\i\in \mathcal{I}_\delta(x)$ we have $K_{\i}\subset B(x,3\delta)$ and $\mu([\i])=r_{\i}^s\gtrsim \delta^s$. It follows that
    \[
        \nu(B(x,3\delta)) \geq \sum_{\i\in \mathcal{I}_\delta(x)} \mu([\i]) \gtrsim \delta^s \#\mathcal{I}_\delta(x)
    \]
    and thus
    \begin{equation}\label{eq:stopfibrecount}
        N_\delta(\Pi^{-1}(x)) \leq \#\mathcal{I}_\delta(x) \lesssim \delta^{-s}\nu(B(x,3\delta)).
    \end{equation}
    Inserting this into~\eqref{eq:fibre-cover-local-dimension} yields
    \[
        \dimub\Pi^{-1}(x) \leq  \limsup_{\delta\to 0} \frac{\log \big( \delta^{-s}\nu(B(x,3\delta)) \big)}{-\log \delta} = s-\underline d_\nu(x).
    \]
\end{proof}

The key ingredient is the dimension conservation principle of Feng and Hu~\cite{FH09}. We record the special case needed here.

\begin{theorem}[{\cite{FH09}}]\label{thm:FH09}
    For any probability vector $\mathbf{p}=(p_i)_{i\in\Lambda}$, $\nu_{\mathbf{p}}$ is exact dimensional and
    \begin{equation}\label{eq:fh09_1}
        \dimh\nu_{\mathbf{p}} = \frac{H_{\mu_{\mathbf{p}}}(\mathcal{P}\mid\sigma^{-1}\Pi^{-1}\mathcal{B}(\R)) - H_{\mu_{\mathbf{p}}}(\mathcal{P}\mid\Pi^{-1}\mathcal{B}(\R))}{-\sum_{i\in\Lambda} p_i\log r_i},
    \end{equation}
    where $\mathcal{P}:=\{[i]:i \in\Lambda\}$ is the first-coordinate partition. Moreover, letting $\{\mu_{\mathbf{p},x}\}_{x\in K}$ denote the conditional measures of $\mu_{\mathbf{p}}$ with respect to the measurable partition of $\Sigma$ into fibres of $\Pi$, we have for $\nu_{\mathbf{p}}$-almost every $x\in K$ that
    \begin{equation}\label{eq:fh09_2}
        \dimh \mu_{\mathbf{p},x} = \dims\nu_{\mathbf{p}} - \dimh\nu_{\mathbf{p}}.
    \end{equation}
\end{theorem}


The first assertion of Proposition~\ref{prop:fibrebound} is a direct consequence of the above two results.

\begin{proof}[Proof of Proposition~\ref{prop:fibrebound}(1)]
    Applying~\eqref{eq:fh09_2} to the natural Bernoulli measure $\mu$, we get 
    \[
        \dimh \mu_{x} = s-\dimh\nu
    \]
    for $\nu$-almost every $x\in K$. Since $\mu_{x}$ is supported on $\Pi^{-1}(x)$, $\dimh\Pi^{-1}(x) \geq s-\dimh\nu$. On the other hand, the exact dimensionality of $\nu$ implies that $\underline{d}_\nu(x)=\dimh\nu$ for $\nu$-almost every $x$. Lemma~\ref{lem:localfibreest} therefore implies that $\dimub\Pi^{-1}(x)\leq s-\dimh\nu$ for $\nu$-almost every $x\in K$. This proves~\eqref{eq:fibredim}. The two stated consequences follow immediately.
\end{proof}

Proposition~\ref{prop:fibrebound}(1) also shows that, if $\dimh\nu<s$ then the set
\[
    E_\infty := \{x\in K:\#\Pi^{-1}(x)=\infty\}
\]
has full $\nu$-measure; so $\dimh E_\infty \geq \dimh\nu$. If additionally $\dimh\nu=\dimh K$ then $\dimh E_\infty = \dimh K$.

\begin{corollary}
    If $\dimh\nu = s$, then 
    \[
        \mu(\{\omega\in\Sigma: \dimb\Pi^{-1}(\Pi(\omega))=0\}) = 1.
    \]
\end{corollary}


To turn from typical fibres to a uniform estimate valid everywhere, the main input is Shmerkin's $L^q$ dimension theorem for self-similar measures satisfying the \emph{exponential separation condition}, a concept introduced by Hochman~\cite{Hoc14}.

\begin{definition}\label{def:esc}
    The IFS $\Phi$ is said to satisfy the exponential separation condition (ESC) if there exist $\delta>0$ and a sequence $n_k\to\infty$ such that $\Delta_{n_k}\geq\delta^{n_k}$, where 
    \[
        \Delta_n := \min\{|\varphi_{\i}(0)-\varphi_{\j}(0)|: \i,\j\in\Lambda^n, \i\neq\j, \lambda_{\i}=\lambda_{\j}\}.
    \]
\end{definition}

We extract the uniform mass estimate needed below.


\begin{lemma}\label{lem:Shm19}
    Suppose $\Phi$ satisfies the ESC and set $t:=\min\{s,1\}$. Then for every $\varepsilon>0$, 
    \begin{equation}\label{eq:lemShm19}
        \nu(B(x,r)) \lesssim_\varepsilon r^{t-\varepsilon}, \quad \forall x\in \R, \forall 0<r<1.
    \end{equation}
\end{lemma}
\begin{proof}
    It suffices to consider $0<\varepsilon<t$. Choose $q>1$ so large that $(1-q^{-1})(t-\varepsilon/2)>t-\varepsilon$. By Shmerkin's $L^q$-dimension theorem~\cite[Theorem 6.6]{Shm19} for general self-similar measures, the dyadic $L^q$ dimension of $\nu$ equals $\min\{s,1\}=t$. Hence
    \[
        \sum_{I\in\mathcal{D}_n} \nu(I)^q \lesssim 2^{-n(q-1)(t-\varepsilon/2)},
    \]
    where $\mathcal{D}_n:=\{[\frac{k}{2^n}, \frac{k+1}{2^n}): k\in\mathbb{Z}\}$. 
    It follows that each $I\in\mathcal{D}_n$ satisfies
    \[
        \nu(I) \lesssim 2^{-n(q-1)(t-\varepsilon/2)/q} \leq 2^{-n(t-\varepsilon)},
    \]
    where the last inequality follows from the choice of $q$. Choose $n$ so that $2^{-(n+1)}<r\leq2^{-n}$. The interval $B(x,r)$ meets only a bounded number of intervals in $\mathcal D_n$, independently of $x$ and $n$. Summing the preceding estimate over these intervals proves~\eqref{eq:lemShm19}.
\end{proof}

The passage from $L^q$ estimates to uniform mass bounds, and from such bounds to global upper box estimates for fibres, is closely related to~\cite[Lemmas 1.7 and 1.8]{Shm19}. We use self-similarity to obtain the local covering estimate needed for an Assouad-dimension bound. Our proof is also motivated by ongoing joint work of the first author with Kong, Li, and Ma~\cite{KLMS26+}. 
Recall that the Assouad dimension of $E\subset\Sigma$ is the infimum of all $\alpha>0$ for which there is a constant $C$ such that $N_r(E\cap B_\rho(\omega,R))\leq C(R/r)^\alpha$ for every $\omega\in E$ and all $0<r<R<1$. 

\begin{lemma}
    Let $\beta:=\max\{s-1,0\}$. If $\Phi$ satisfies the ESC, then for every $\varepsilon>0$ there exists a constant $C_\varepsilon>0$ such that
    \[
        N_r\big( \Pi^{-1}(x)\cap B_\rho(\omega,R) \big) \leq C_\varepsilon \Big(\frac{R}{r}\Big)^{\beta+\varepsilon}
    \]
    for every $x\in K$, every $\omega\in\Pi^{-1}(x)$ and all $0<r<R<1$.
\end{lemma}
\begin{proof}
    Let $t=\min\{s,1\}$ be as in Lemma~\ref{lem:Shm19}. So $s-t=\max\{s-1,0\}=\beta$. By~\eqref{eq:stopfibrecount} and~\eqref{eq:lemShm19}, for every $x\in K$ and $0<\delta<1$ we have (uniformly in $x\in K$) 
    \begin{equation}\label{eq:unilocalest}
        N_\delta(\Pi^{-1}(x)) \lesssim \delta^{-s}\nu(B(x,3\delta)) \lesssim_\varepsilon \delta^{-s}\delta^{t-\varepsilon} = \delta^{-(\beta+\varepsilon)}.
    \end{equation}
    To localize the estimate, fix $x\in K$, $\omega\in\Pi^{-1}(x)$, and $0<r<R<1$. Let $\i$ be the unique prefix of $\omega$ in $\Lambda_R$. So $B_\rho(\omega,R)\subset [\i]$. Since $x=\Pi(\omega)\in K_{\i}$, the point $y:=\varphi_{\i}^{-1}(x)$ belongs to $K$, and
    \begin{equation}\label{eq:piinvxcapi}
        \Pi^{-1}(x)\cap[\i]=\{\i\eta: \eta\in\Pi^{-1}(y)\}
    \end{equation}
    Note that 
    \begin{equation}\label{eq:prefixsimilarity}
        \rho(\i\eta, \i\eta') = r_{\i}\rho(\eta,\eta'), \quad \forall \eta,\eta'\in\Sigma.
    \end{equation}
    In particular, the cylinder $[\i]$ has $\rho$-diameter $r_{\i}$. 
    
    If $r\geq r_{\i}$, then $N_r(\Pi^{-1}(x)\cap[\i])=1\leq (R/r)^{\beta+\varepsilon}$. If $r<r_{\i}$ then
    \begin{align*}
        N_r( \Pi^{-1}(x)\cap [\i] ) &= N_{r/r_{\i}}(\Pi^{-1}(y)) && \text{(by~\eqref{eq:piinvxcapi}) and~\eqref{eq:prefixsimilarity}} \\
        &\lesssim_\varepsilon \Big( \frac{r_{\i}}{r} \Big)^{\beta+\varepsilon} && \text{(by~\eqref{eq:unilocalest})} \\
        &\leq \Big( \frac{R}{r} \Big)^{\beta+\varepsilon} && \text{(since $\i\in\Lambda_R$)}
    \end{align*}
    and the lemma is proved.
\end{proof}

Since the above estimate holds for every $\varepsilon>0$, it follows that
\[
    \dim_{\textup{A}} \Pi^{-1}(x) \leq\beta=\max\{s-1,0\} \quad\text{for every } x\in K.
\]
This settles Proposition~\ref{prop:fibrebound}(2).

\subsection{Measure-theoretic separation}

Here we prove the measure-theoretic separation and density statements from Proposition~\ref{thm:singularpieces} separately.

\begin{proof}[Proof of Proposition~\ref{thm:singularpieces}(1)]
    Fix any probability vector $\mathbf{p}=(p_i)_{i\in\Lambda}$ with $p_i>0$. Since $\mu_{\mathbf{p}}$ is Bernoulli, $\mathcal{P}$ is independent of $\sigma^{-1}\Pi^{-1}\mathcal{B}(\R)$. So 
    \[
        H_{\mu_{\mathbf{p}}}(\mathcal{P}\mid\sigma^{-1}\Pi^{-1}\mathcal{B}(\R)) = H_{\mu_{\mathbf{p}}}(\mathcal{P}) = -\sum_{i\in\Lambda} p_i\log p_i.
    \]
    The assumption $\dimh\nu_{\mathbf{p}}=\dims\nu_{\mathbf{p}}$ together with the formula~\eqref{eq:fh09_1} then implies
    \begin{equation}\label{eq:H=0}
        H_{\mu_{\mathbf{p}}}(\mathcal{P}\mid\Pi^{-1}\mathcal{B}(\R)) = 0.
    \end{equation}
    For $i\in\Lambda$, define $\nu_i:=p_i\varphi_i\nu_{\mathbf{p}}$. Then $\nu_{\mathbf{p}}=\sum_{i\in\Lambda} \nu_i$ and hence $\nu_i\ll\nu_{\mathbf{p}}$. Let $f_i:=\frac{\rd\nu_i}{\rd\nu_{\mathbf{p}}}$. Then $f_i\geq 0$ and $\sum_{i\in\Lambda} f_i=1$ $\nu_{\mathbf{p}}$-almost everywhere.

    We claim that for every $i$,
    \[
        f_i(\Pi(\omega)) = \mathbb{E}_{\mu_{\mathbf{p}}}\bigl( \mathbbm{1}_{[i]}\mid\Pi^{-1}\mathcal{B}(\R^d) \bigr)(\omega), \quad \mu_{\mathbf{p}}\text{-a.e. } \omega\in\Sigma.
    \]
    To this end, it suffices to show that 
    \[
        \int_{\Pi^{-1}(A)} f_i\circ\Pi \rd\mu_{\mathbf{p}} = \int_{\Pi^{-1}(A)} \mathbbm{1}_{[i]} \rd\mu_{\mathbf{p}}, 
    \]
    where $A\subset\R$ is Borel. For such $A$,
    \begin{align*}
        \int_{\Pi^{-1}(A)} \mathbbm{1}_{[i]} \rd\mu_{\mathbf{p}} &= \mu_{\mathbf{p}}([i]\cap \Pi^{-1}(A)) \\
        &= \mu_{\mathbf{p}}(\{\omega: \omega=i\tau \text{ for some } \tau\in\Sigma \text{ and } \Pi(\omega)\in A\}) \\
        &= p_i\mu_{\mathbf{p}}(\{\tau\in\Sigma: \varphi_i(\Pi(\tau))\in A\}) \\
        &= p_i\mu_{\mathbf{p}}\bigl( \Pi^{-1}\varphi_i^{-1}(A) \bigr) \\
        &= p_i\varphi_i\nu_{\mathbf{p}}(A) \\ 
        &= \nu_i(A) \\
        &= \int_A f_i \rd\nu_{\mathbf{p}} \\
        &= \int_{\Pi^{-1}(A)} f_i\circ\Pi \rd\mu_{\mathbf{p}},
    \end{align*}
    which proves the claim.

    By~\eqref{eq:H=0} and the definition of conditional entropy, we obtain
    \begin{align*}
        0 &= H_{\mu_{\mathbf{p}}}(\mathcal{P}\mid\Pi^{-1}\mathcal{B}(\R)) \\
        &= -\sum_{i\in\Lambda} \int \mathbb{E}_{\mu_{\mathbf{p}}}\big( \mathbbm{1}_{[i]}\mid\Pi^{-1}\mathcal{B}(\R) \big)\log \mathbb{E}_{\mu_{\mathbf{p}}}\big( \mathbbm{1}_{[i]}\mid\Pi^{-1}\mathcal{B}(\R) \big) \rd\mu_{\mathbf{p}} \\
        &= -\sum_{i\in\Lambda} \int_\Sigma (f_i\circ\Pi)\log (f_i\circ\Pi) \rd\mu_{\mathbf{p}} \\
        &= \int_K -\sum_{i\in\Lambda} f_i\log f_i \rd\nu_{\mathbf{p}}.
    \end{align*}
    The integrand is non-negative, so it vanishes $\nu_{\mathbf{p}}$-almost everywhere.
    Since $(f_i(x))_{i\in\Lambda}$ is a probability vector, its entropy vanishes if and only if it is a vertex of the probability simplex. Consequently, for $\nu_{\mathbf{p}}$-almost every $x$ there exists a unique index $i=i(x)$ such that $f_i(x)=1$ and $f_j(x)=0$ for all $j\neq i$. Define 
    \[
        B_i:=\{x\in K: f_i(x)=1 \text{ and } f_j(x)=0 \text{ for } j\neq i\}. 
    \]
    Since every $f_i$ is Borel, the sets $B_i$ are Borel and pairwise disjoint, and their union has full $\nu_{\mathbf{p}}$-measure. Moreover, 
    \[
        \nu_i(B_i^c) = \int_{B_i^c} f_i\rd\nu_{\mathbf{p}} = 0. 
    \]
    So $\nu_i$ is concentrated on $B_i$ and $\nu_i\perp\nu_j$ when $i\neq j$. Finally, since multiplication by a positive scalar preserves singularity, $\varphi_i\nu_{\mathbf{p}}\perp\varphi_j\nu_{\mathbf{p}}$ whenever $i\neq j$.
\end{proof}

To prove the quantitative density estimate in Proposition~\ref{thm:singularpieces}(2), note that there exists a uniform constant $c_0>0$ such that
\begin{equation}\label{eq:uniforssmass}
    \nu(B(z,r)) \geq c_0r^s, \quad \forall z\in K, 0<r<1
\end{equation}
because every such ball contains a cylinder set (along a coding of $z$) whose diameter is at most $r$. The role of~\eqref{eq:uniforssmass} is that every additional first level piece containing a point contributes a definite amount to its normalized local mass.

\begin{proof}[Proof of Proposition~\ref{thm:singularpieces}(2)]
    For $y\in\R$ and $r>0$, write $d(y,r):=\frac{\nu(B(y,r))}{r^s}$. Let $\bi=i_1i_2\cdots\in\Sigma$ and let $x=\Pi(\bi)$. For any $j\in\Lambda$, if $x\in K_j$ then $\varphi_{j}^{-1}(x)\in K$ and~\eqref{eq:uniforssmass} gives 
    \begin{equation}
        \nu(B(\varphi_j^{-1}(x),r))\geq c_0r^s, \quad 0<r<1.
    \end{equation}
    So for all small $\delta>0$, the self-similarity of $\nu$ gives 
    \begin{align*}
        d(\Pi(\bi),\delta) 
        &= \frac{\sum_{j\in\Lambda} r_j^s\nu(B(\varphi_j^{-1}(x), \delta/r_{j}))}{\delta^s} \\
        &= \frac{r_{i_1}^s}{\delta^s}\cdot\nu(B(\varphi_{i_1}^{-1}(x),\delta/r_{i_1})) + \sum_{j\in\Lambda\setminus\{i_1\}} \frac{r_j^s}{\delta^s}\cdot \nu(B(\varphi_j^{-1}x, \delta/r_j)) \\
        &\geq d(\varphi_{i_1}^{-1}(x),\delta/r_{i_1}) + c_0\#\{j\neq i_1: x\in K_j\} \\
        &\geq d(\Pi(\sigma\bi),\delta/r_{i_1}) + c_0\mathbbm{1}_{\mathcal{C}}(\Pi(\bi)).
    \end{align*} 
    Taking $\delta=r_{i_1\cdots i_n}\cdot\min_i r_i$ and iterating, we obtain 
    \[
        d(\Pi(\bi), r_{i_1\cdots i_n}) \geq c_0\sum_{k=0}^{n-1} \mathbbm{1}_{\mathcal{C}}(\Pi(\sigma^k\bi)).
    \]
    By Birkhoff's ergodic theorem, for $\mu$-almost every $\bi\in\Sigma$, 
    \[
        \lim_{n\to\infty} \frac{1}{n}\sum_{k=0}^{n-1} \mathbbm{1}_{\mathcal{C}}(\Pi(\sigma^k\bi)) = \mu(\Pi^{-1}(\mathcal{C})) = \nu(\mathcal{C}) > 0.
    \]
    So there exists a constant $c_1>0$ such that for $\mu$-almost every $\bi\in\Sigma$, 
    \begin{equation}\label{eq:Birkhoffest}
        \sum_{k=0}^{n-1} \mathbbm{1}_{\mathcal{C}}(\Pi(\sigma^k\bi)) \geq c_1n \quad \text{ for all large } n.
    \end{equation}

    Fix any such $\bi\in\Sigma$ and let $0<r<1$ be small. Pick $n$ with $r_{i_1\cdots i_{n}}<r\leq r_{i_1\cdots i_{n-1}}$ (so $r\approx r_{i_1\cdots i_n}$). By monotonicity of balls and the previous observations,
    \[
        \nu(B(\Pi(\bi),r)) \geq \nu(B(\Pi(\bi), r_{i_1\cdots i_{n}})) \geq c_0c_1 nr_{i_1\cdots i_n}^s \gtrsim nr^s.
    \] 
    Since the contraction ratios range over a finite set,
    \[
        \log r^{-1} \approx \log r_{i_1\cdots i_n}^{-1} = \sum_{k=1}^n \log r_{i_k}^{-1} \approx n.
    \]
    We conclude that $\nu(B(\Pi(\bi),r))\gtrsim r^s\log r^{-1}$ for all sufficiently small $r$, which completes the proof.
\end{proof}

\subsection{A full-dimensional critical set under the ESC}

Consider the IFS 
\[
    \Phi = \Big\{ \varphi_0(x) = \frac{x}{4}, \varphi_u(x)=\frac{x}{4}+u, \varphi_1(x)=\frac{x}{4}+1 \Big\},
\]
where $u:=\sum_{k=0}^\infty 4^{-2^k}$ is an irrational number. Let $K$ be its attractor and let $s:=\frac{\log 3}{\log 4}$ be its similarity dimension. We prove the three claims in Proposition~\ref{exa:escwithnddcri} separately.

\begin{lemma}
    The IFS $\Phi$ satisfies the ESC.
\end{lemma}
\begin{proof}
    For $m\geq 0$, write $\sum_{k=0}^m 4^{-2^k}=\frac{p_m}{q_m}$, where $q_m:=4^{2^m}$ and 
    \[
        p_m = \sum_{k=0}^{m} 4^{2^m-2^k} \equiv 1 \pmod 4.
    \]
    So $\frac{p_m}{q_m}$ is reduced. Moreover,
    \begin{equation}\label{eq:u-um_est}
        0< u-\frac{p_m}{q_m} = \sum_{k=m+1}^\infty 4^{-2^k} < \sum_{n=1}^\infty (4^{-2^{m+1}})^n = \frac{q_{m+1}^{-1}}{1-q_{m+1}^{-1}} \leq \frac{16}{15}q_{m+1}^{-1} < \frac{2}{q_m^2}.
    \end{equation}

    We claim that
    \begin{equation}\label{eq:claiminexa_esc}
        \Big| u-\frac{p}{q} \Big| > \frac{1}{32q^3}, \quad \forall p\in\mathbb{Z}, q\in\mathbb{Z}^+.
    \end{equation}
    It suffices to prove this for reduced $p/q$, since reducing the denominator only makes the claimed lower bound larger. Choose $m\geq 1$ with $q_{m-1}\leq 4q<q_m$. Since $\frac{p_m}{q_m}$ and $\frac{p}{q}$ are reduced and $q_m>q$, they are distinct fractions. Therefore 
    \[
        \Big| \frac{p}{q}-\frac{p_m}{q_m} \Big| \geq \frac{1}{qq_m}.
    \]
    Using~\eqref{eq:u-um_est} and $q_m=q^2_{m-1}$ (so $q_{m}\leq 16q^2$), we get
    \begin{align*}
        \Big| u-\frac{p}{q} \Big| \geq \Big| \frac{p}{q}-\frac{p_m}{q_m} \Big| - \Big| u-\frac{p_m}{q_m} \Big| 
        > \frac{1}{qq_m} - \frac{2}{q_m^2} 
        > \frac{1}{2qq_m} 
        \geq \frac{1}{32q^3},
    \end{align*}
    as desired.

    Let $n\geq 1$ and take two distinct words $\i=i_1\cdots i_n$, $\j=j_1\cdots j_n\in\{0,u,1\}^n$. Their images differ by $\varphi_{\i}(0) - \varphi_{\j}(0) = \sum_{k=1}^n (i_k-j_k)4^{-k+1}$. One can write $i_k-j_k=\alpha_k+\beta_ku$, where $\alpha_k, \beta_k\in\{0,\pm 1\}$. Multiplying by $4^{n-1}$, we obtain 
    \begin{equation}\label{eq:exa_esc_distance}
        \varphi_{\i}(0) - \varphi_{\j}(0) = 4^{-n+1}(A+Bu),
    \end{equation}
    where $A:=\sum_{k=1}^n \alpha_k 4^{n-k}$ and $B:=\sum_{k=1}^n \beta_k4^{n-k}$ are both integers with absolute value no more than $\sum_{k=0}^{n-1}4^k<4^n$. 

    Note that exact overlaps never occur. Indeed, if $A+Bu=0$, then since $u$ is irrational, $A=B=0$. But if $A=0$ then all $\alpha_k$ must vanish. Similarly $B=0$ forces all $\beta_k=0$. Hence $i_k=j_k$ for all $k$, so $\i=\j$.

    Let us estimate~\eqref{eq:exa_esc_distance}. If $B\neq 0$, then~\eqref{eq:claiminexa_esc} with $q=|B|$ gives
    \[
        |A+Bu| = |B|\Big| u+\frac{A}{B} \Big| > \frac{1}{32|B|^2} > \frac{1}{32\cdot 4^{2n}}
    \]
    and thus $|\varphi_{\i}(0) - \varphi_{\j}(0)|>\frac{1}{8}\cdot 4^{-3n}$. If $B=0$, then $|A|\geq 1$ and $|\varphi_{\i}(0) - \varphi_{\j}(0)| = 4^{-n+1}|A| \geq 4^{-n+1}$. Thus for all $n\ge1$ and all distinct words $\i,\j$ of length $n$,
    \[
        |\varphi_{\i}(0)-\varphi_{\j}(0)| \geq \frac{1}{8}\cdot 4^{-3n} = \frac{1}{8}\cdot 64^{-n}.
    \]
    So $\Phi$ satisfies the ESC.
\end{proof}

\begin{lemma}
    We have $\dimh (\varphi_0(K)\cap\varphi_u(K)) = s = \dimh K$.
\end{lemma}
\begin{proof}
    Let $I:=\{2^k: k\geq 0\}$ and define 
    \[
        \Sigma_I := \{\omega = \omega_1\omega_2\cdots\in\{0,u,1\}^{\mathbb{N}}: \omega_k=1 \text{ if } k\in I\}.
    \]
    Let $M_n:=\#(I\cap\{1,\ldots,n\})=\lfloor \log_2n \rfloor+1$. Note that there are $3^{n-M_n}$ admissible prefixes of length $n$ in $\Sigma_I$.

    For $\omega\in\Sigma_I$, since $\omega_k=1$ when $k\in I$, 
    \begin{align*}
        \varphi_0(\Pi(\omega)) = \frac{1}{4}\Pi(\omega) =  \sum_{k=1}^\infty \omega_k4^{-k} = \sum_{k\notin I} \omega_k4^{-k} + \sum_{k\in I} 4^{-k} = \sum_{k\notin I} \omega_k4^{-k} + u.
    \end{align*}
    Define $y:=\sum_{k=1}^\infty \eta_k4^{-k+1}$, where $\eta_k=0$ if $k\in I$ and $\eta_k=\omega_k$ otherwise. Then $y\in K$ and 
    \[
        \varphi_0(\Pi(\omega)) = \frac{1}{4}y+u \in\varphi_u(K).
    \]
    Thus $\varphi_0(\Pi(\Sigma_I))\subset\varphi_0(K)\cap\varphi_u(K)$ and it suffices to prove that 
    \[
        \dimh\Pi(\Sigma_I)=s.
    \] 

    To this end, fix any $0<t<s$ and pick $\alpha$ with $t<\alpha<s$. Let $\mu$ be the uniform Bernoulli measure on $\Sigma=\{0,u,1\}^{\mathbb{N}}$ and let $\nu=\Pi\mu$. Since $\Phi$ satisfies the ESC and $s<1$, we can find by Lemma~\ref{lem:Shm19} (or the Frostman bound derived therefrom; see~\cite[Theorem 6.2]{Shm19}) a constant $C_\alpha$ such that
    \begin{equation}\label{eq:exa_escfrostman}
        \nu(B(x,r)) \leq C_\alpha r^\alpha, \quad \forall x\in\R, 0<r\leq 1.
    \end{equation}
    Construct a product measure $\tilde{\mu}$ on $\Sigma_I$ by placing mass $\delta_1$ on coordinates $k\in I$ and the uniform distribution on $\{0,u,1\}$ otherwise. Then $\tilde{\nu}:=\Pi\tilde{\mu}$ is supported on $\Pi(\Sigma_I)$.

    Let $x\in\R$ and let $0<r<1$. Pick $n$ with $\frac{4}{3}4^{-n}\leq r<\frac{4}{3}4^{-n+1}$, where $\frac{4}{3}$ is simply the diameter of $K$. Let 
    \[
        \mathcal{W}_n(x,r) := \{\i\in\{0,u,1\}^n: \Pi([\i]\cap\Sigma_I) \cap B(x,r) \neq \varnothing\}.
    \]
    For every admissible prefix $\i$ of length $n$, $\tilde{\mu}([\i])=3^{-(n-M_n)}$. So  
    \begin{equation}\label{eq:exa_esc_newmass}
        \tilde{\nu}(B(x,r)) \leq \#\mathcal{W}_n(x,r) \cdot 3^{-(n-M_n)}.
    \end{equation}
    On the other hand, if $\i\in\mathcal{W}_n(x,r)$ then $\Pi([\i])\cap B(x,r)\neq\varnothing$. Since $\Pi([\i])$ has diameter $\frac{4}{3}4^{-n}\leq r$, $\Pi([\i])\subset B(x,2r)$. Therefore
    \begin{equation}\label{eq:exa_esc_mass}
        \nu(B(x,2r)) \geq \mu\Big( \bigcup_{\i\in\mathcal{W}_n(x,r)} [\i] \Big) = \#\mathcal{W}_n(x,r)\cdot 3^{-n}.
    \end{equation}
    Combining~\eqref{eq:exa_escfrostman}, ~\eqref{eq:exa_esc_newmass}, ~\eqref{eq:exa_esc_mass} and $M_n\leq\log_2n+1$, we get 
    \[
        \tilde{\nu}(B(x,r)) \leq 3^{M_n}\nu(B(x,2r)) \lesssim_\alpha 3^{M_n}r^\alpha \lesssim n^{\log_23}r^\alpha = n^{\log_2 3}r^{\alpha-t}r^t. 
    \]
    Because $r\approx 4^{-n}$, the factor $n^{\log_2 3}r^{\alpha-t}$ is bounded uniformly in $n$.  
    Hence
    \[
        \tilde{\nu}(B(x,r)) \lesssim_{\alpha,t} r^t, \quad \forall 0<r<1,
    \]
    with a constant independent of $x$. By Frostman's lemma, $\dimh\Pi(\Sigma_I)\geq t$. Since $t<s$ was arbitrary, $\dimh\Pi(\Sigma_I)\geq s$. The reverse inequality follows from $\Pi(\Sigma_I)\subset K$ and $\dimh K=s$.
\end{proof}


\begin{lemma}
    There exists a point $x_*\in K$ with $\#\Pi^{-1}(x_*)\geq 2^{\aleph_0}$.
\end{lemma}
\begin{proof}
    Since $u=\sum_{k=0}^\infty 4^{-2^k}$, we have 
    \begin{equation}\label{eq:ureplace}
        u4^{-n+1} = \sum_{k=0}^\infty 4^{-(n+2^k-1)}, \quad \forall n\geq 1.
    \end{equation}
    This identity naturally allows one occurrence of the symbol $u$ to be replaced by infinitely many occurrences of the symbol $1$. To make this precise, let $H:=\{0\}\cup\{2^k: k\geq 0\}$.

    We first choose a sequence $0<n_1<n_2<\cdots$ such that the translated copies $H+n_j$ are pairwise disjoint. Consider the difference set $H-H:=\{a-b: a,b\in H\}$. Recall that $\#(H\cap [0,N])=O(\log N)$. So
    \[
        \#\big( (H-H)\cap [-N,N] \big) = O((\log N)^2).
    \]
    Then we can select the $n_j$'s recursively: once $n_1,\ldots,n_p$ have been chosen, we may choose $n_{p+1}>n_p$ outside the union 
    \[
        \bigcup_{j=1}^p (H-H+n_j) := \bigcup_{j=1}^p \{a-b+n_j: a,b\in H\}
    \]
    because it contains only $O(p(\log N)^2)$ integers in $[1,N]$. We then have $n_{p+1}-n_j\notin H-H$ for all $1\leq j\leq p$. Continuing in this manner, the sequence $\{n_j\}$ can be found.

    For every binary sequence $\varepsilon=(\varepsilon_j)_{j=1}^\infty\in\{0,1\}^{\mathbb{N}}$, define a coding $\omega^\varepsilon\in\{0,u,1\}^{\mathbb{N}}$ as follows. For $j\geq 1$, if $\varepsilon_j=0$ then set $\omega^\varepsilon_{n_j}=u$ and $\omega^\varepsilon_{n_j+2^k}=0$ for all $k\geq 0$; if $\varepsilon_j=1$ then set $\omega^\varepsilon_{n_j}=0$ and $\omega^\varepsilon_{n_j+2^k}=1$ for all $k\geq 0$; put all $\omega^\varepsilon_i=0$ at all remaining positions. The disjointness of $n_j+H$ guarantees that all $\omega^\varepsilon$'s are well defined. 
    
    We now verify explicitly that the projection does not depend on $\varepsilon$. Since every nonzero coordinate of $\omega^\varepsilon$ belongs to one of the pairwise disjoint sets $n_j+H$, absolute convergence allows us to group the coding series according to these sets:
    \begin{align*}
        \Pi(\omega^\varepsilon) = \sum_{m=1}^{\infty} \omega^\varepsilon_m4^{-(m-1)}  = \sum_{j=1}^{\infty} \Big( \omega^\varepsilon_{n_j}4^{-(n_j-1)} + \sum_{k=0}^{\infty} \omega^\varepsilon_{n_j+2^k} 4^{-(n_j+2^k-1)} \Big).
    \end{align*}
    For each $j$, if $\varepsilon_j=0$ then the expression in parentheses is $u4^{-(n_j-1)}=u4^{-n_j+1}$; if $\varepsilon_j=1$ then it is $\sum_{k=0}^{\infty}4^{-(n_j+2^k-1)} =u4^{-n_j+1}$ by~\eqref{eq:ureplace}. Thus the contribution of every block
    $n_j+H$ is $u4^{-n_j+1}$ in either case. Consequently,
    \[
        \Pi(\omega^\varepsilon) = \sum_{j=1}^{\infty}u4^{-n_j+1} =: x_*,
    \]
    independently of $\varepsilon$.
    
    Moreover, if $\varepsilon\neq\varepsilon'$ then $\varepsilon_j\neq\varepsilon'_j$ for some $j$ and therefore $\omega^\varepsilon_{n_j}\neq\omega^{\varepsilon'}_{n_j}$. Thus $\omega^\varepsilon$, $\varepsilon\in\{0,1\}^{\mathbb{N}}$ are pairwise distinct; in particular, $\#\Pi^{-1}(x_*)\geq\#\{0,1\}^{\mathbb{N}}=2^{\aleph_0}$. Since $\Sigma$ itself has cardinality $2^{\aleph_0}$, equality follows.
\end{proof}

\section{Transversal families: critical sets and infinite fibres}
\label{sec:transversalfam}

Throughout this section, we fix a bounded open set $U\subset\R^d$ and consider the family
\[
    \Phi^{\t} :=\{\varphi^{\t}_i(x) = \lambda_i(\t)x+d_i(\t)\}_{i\in\Lambda}, \quad \t\in U,
\]
where $\lambda_i(\t)\in(-1,1)\setminus\{0\}$ and $d_i(\t)\in\R$ are both of class $C^1$ on $U$. Denote by $K^{\t}$ the attractor and by $s(\t)$ the similarity dimension of $\Phi^{\t}$. The associated coding map is
\[
    \Pi^{\t}(\bi) = \lim_{n\to\infty} \varphi^{\t}_{i_1}\circ\cdots\circ\varphi^{\t}_{i_n}(0), \quad \forall \bi=i_1i_2\cdots\in\Sigma.
\]

We record the uniform $C^1$ dependence of the coding maps on the parameter.

\begin{lemma}\label{lem:uniformc1}
    Let $Q\subset U$ be compact. There exists $C>0$ and $0<\theta<1$ such that, 
    \[
        \|\Pi^{(\cdot)}(\bi)-\Pi^{(\cdot)}(\bj)\|_{C^1(Q)} \leq C\theta^{|\bi\wedge\bj|}.
    \]
    In particular, the family $\{\t\mapsto \Pi^{\t}(\bi): \bi\in\Sigma\}$ is uniformly bounded in $C^1(Q)$, and the family $\{\t\mapsto \nabla\Pi^{\t}(\bi): \bi\in\Sigma\}$ is equicontinuous on $Q$.
\end{lemma}
\begin{proof}
    Set $\rho_Q := \max_{i\in\Lambda}\sup_{\t\in Q}|\lambda_i(\t)|<1$. 
    Since $\Lambda$ is finite and the functions $\lambda_i, d_i$ are $C^1$, $\lambda_i, d_i$ and their first derivatives are uniformly bounded on $Q$. For $\bi=i_1i_2\cdots\in\Sigma$, 
    \[
        \Pi^{\t}(\bi) = \sum_{k=1}^\infty d_{i_k}(\t)\prod_{\ell=1}^{k-1} \lambda_{i_{\ell}}(\t).
    \]
    The series converges uniformly on $Q\times\Sigma$. Moreover, 
    \[
        \Big| \nabla\prod_{\ell=1}^{k-1} \lambda_{i_{\ell}}(\t) \Big| \lesssim_Q (k-1)\rho_Q^{k-2}
    \]
    and hence 
    \[
        \Big| \nabla\Big( d_{i_k}(\t)\prod_{\ell=1}^{k-1} \lambda_{i_{\ell}}(\t) \Big) \Big| \lesssim_Q \rho_Q^{k-1}+(k-1)\rho_Q^{k-2}.
    \]
    Since the right hand side is summable in $k$, the series can be differentiated term by term, uniformly on $Q\times\Sigma$. In particular, 
    \[
        \sup_{\bi\in\Sigma} \|\Pi^{(\cdot)}(\bi)\|_{C^1(Q)} < \infty.
    \]

    Similarly, one can obtain uniformly for $\t\in Q$ that if $\bi$ and $\bj$ agree in their first $n$ symbols, then $|\Pi^{\t}(\bi)-\Pi^{\t}(\bj)| \lesssim_Q \rho_Q^n$ 
    and
    \[
        \big| \nabla(\Pi^{\t}(\bi)-\Pi^{\t}(\bj))\big| \lesssim_Q (n+1)\rho_Q^{n-1}.
    \]
    Since $\rho_Q<1$, we may choose $\theta\in(\rho_Q,1)$ so that $(n+1)\rho_Q^{n-1}\lesssim \theta^n$ for all $n$, which gives the desired $C^1$ estimate.

    For the equicontinuity of the gradients, let 
    \[
        \Pi^{\t}_N(\bi) := \sum_{k=1}^N d_{i_k}(\t)\prod_{\ell=1}^{k-1} \lambda_{i_{\ell}}(\t).
    \]
    The estimates above imply 
    \[
        \lim_{N\to \infty} \Big( \sup_{\bi\in\Sigma}\sup_{\t\in Q} \big|\nabla(\Pi^{\t}(\bi)-\Pi^{\t}_N(\bi))\big| \Big) = 0.
    \]
    For each $N$, the function $\nabla\Pi^{\t}_N(\bi)$ depends only on the first $N$ symbols of $\bi$. So as $\bi$ varies, there are only finitely many such functions and each of them is continuous on $Q$. Combining this with the uniform convergence of the gradients, we see that the family $\{\t\mapsto \nabla\Pi^{\t}(\bi): \bi\in\Sigma\}$ is equicontinuous on $Q$.
\end{proof}

Similarly as in the determined case, for a finite word $\i=i_1\cdots i_p$ we write
\[
    \lambda_{\i}^{\t} := \lambda_{i_1}(\t)\cdots \lambda_{i_p}(\t) \quad\text{and}\quad \varphi_{\i}^{\t} := \varphi_{i_1}^{\t}\circ\cdots\circ\varphi_{i_p}^{\t}.
\]


\begin{definition}\label{de:trans}
    The family $\Phi^{\t}$ is said to satisfy the \emph{transversality condition} if there exists a constant $\alpha>0$ such that for all $\bi,\bj\in\Sigma$ with $i_1\neq j_1$ and all $\t\in U$, 
    \begin{equation}\label{eq:transversality}
        |\Pi^{\t}(\bi)-\Pi^{\t}(\bj)|<\alpha \quad\Longrightarrow\quad \big\| \nabla\bigl( \Pi^{\t}(\bi)-\Pi^{\t}(\bj) \bigr) \big\|>\alpha.
    \end{equation}
\end{definition}

We shall occasionally use the following standard consequence of transversality. For a proof, see~\cite[Chapter 6]{BSS23}.

\begin{lemma}
    Suppose $\{\Phi^{\t}\}_{\t\in U}$ satisfies the transversality condition. For any compact $Q\subset U$, there exist constants $C_Q>0$, $\delta_Q>0$ such that for every $\bi,\bj\in\Sigma$ with $i_1\neq j_1$, 
    \[
        \mathcal{L}^d(\{\t\in Q: |\Pi^{\t}(\bi)-\Pi^{\t}(\bj)|<\delta\}) \leq C_Q\delta, \quad \forall 0<\delta\leq\delta_Q.
    \]
\end{lemma}

\subsection{Density and dimension of infinite-fibre parameters}

We first prove Theorem~\ref{thm:denseinffibre}. Let $\widetilde{U}$ and $T$ be as in Theorem~\ref{thm:denseinffibre}. It is clear from the definition that $T\subset\widetilde{U}$. We approximate the Hausdorff dimension of $K^{\t_0}$ from below using SSC subsystems. Transversality yields a family of $(d-1)$-dimensional coincidence surfaces, indexed by the symbolic space of a subsystem. The SSC bounds the distances between these surfaces from below in terms of the symbolic distance. A Frostman argument then adds the symbolic dimension to this codimension-one parameter dimension.

\begin{lemma}\label{lem:sscsubsystem}
    Let $\t_0\in U$. For every $\varepsilon>0$, there exists a finite set $\Gamma\subset\Lambda^*$ such that the subsystem $\Psi^{\t_0}:=\{\varphi_{\i}^{\t_0}: \i\in\Gamma\}$ satisfies the SSC and has similarity dimension at least $\dimh K^{\t_0}-\varepsilon$. Consequently, the map $\t\mapsto\dimh K^{\t}$ is lower semicontinuous on $U$.
\end{lemma}
\begin{proof}
    The first assertion can be found in~\cite{Far19}. The second one is also well known; we present the short proof here for completeness. Fix $\varepsilon>0$ and choose $\Gamma$ accordingly. Since $\Psi^{\t_0}$ satisfies the SSC, the same remains true for $\Psi^{\t}$ whenever $\t$ is close enough to $\t_0$. Under the SSC, the Hausdorff dimension of the attractor coincides with the similarity dimension, which depends continuously on $\t$. So for all $\t$ sufficiently close to $\t_0$ we have 
    \[
        \dimh K^{\t} \geq \dimh K^{\t_0}-2\varepsilon.
    \]
    Since $\varepsilon>0$ was arbitrary, lower semicontinuity follows.
\end{proof}

The following elementary dimension estimate will be applied to the family of coincidence surfaces.

\begin{lemma}\label{lem:hypersurface-family}
    Let $(\Omega,\rho_*)$ be a compact metric space carrying an $s$-Frostman probability measure $\nu$. Assume that to each $\omega\in\Omega$ there corresponds a unique non-empty compact set $E_\omega\subset U$ such that:
    \begin{enumerate}
        \item $\h^{d-1}(E_\omega)\approx 1$ and $\h^{d-1}(B(\t,r)\cap E_\omega) \lesssim r^{d-1}$ uniformly in $\omega$, $\t$ and $0<r<1$;
        \item $\dist(E_\omega,E_\eta)\gtrsim \rho_*(\omega,\eta)$ for all $\omega,\eta\in\Omega$;
        \item $\omega\mapsto \h^{d-1}(E_\omega\cap A)$ is Borel for every Borel $A\subset\R^d$.
    \end{enumerate}
    Then
    \[
        \dimh \Bigl(\bigcup_{\omega\in\Omega}E_\omega\Bigr)\geq (d-1)+s.
    \]
\end{lemma}
\begin{proof}
    Define a finite Borel measure $\mu$ by
    \[
        \mu(A) := \int_\Omega \mathcal H^{d-1}(A\cap E_\omega)\,d\nu(\omega).
    \]
    By (1), $\mu(\R^d)\approx 1>0$. Fix $\t\in U$ and $0<r<1$. Let 
    \[
        \Omega(\t,r) := \{\omega\in\Omega: E_\omega\cap B(\t,r)\neq\varnothing\}.
    \]
    If this set is empty then $\mu(B(\t,r))=0$. Otherwise, choose $\omega_0\in\Omega(\t,r)$. Then every $\omega\in\Omega(\t,r)$ satisfies $\dist(E_\omega,E_{\omega_0}) \leq 2r$. By the separation hypothesis (2), $\rho_*(\omega,\omega_0)\lesssim r$. Hence there exists a constant $C>0$ such that $\Omega(\t,r)\subset B_{\rho_*}(\omega_0,Cr)$. It then follows from the Frostman property of $\nu$ that
    \[
        \mu(B(\t,r)) \leq \int_{B_{\rho_*}(\omega_0,Cr)} \mathcal H^{d-1}(E_\omega\cap B(\t,r))\,d\nu(\omega) \lesssim r^{d-1}\cdot r^s = r^{d-1+s}.
    \]
    By the mass distribution principle,
    \[
        \dim_H\Bigl(\bigcup_{\omega\in\Omega}E_\omega\Bigr)\geq  d-1+s.
    \]
\end{proof}

The separation assumption in the above lemma will follow from the next observation.

\begin{lemma}\label{lem:zeroseparation}
    Let $(\Omega,\rho_*)$ be a compact metric space and let $Q\subset\R^d$ be compact. Assume that $h$ and $\{f_\omega\}_{\omega\in\Omega}$ are uniformly Lipschitz on $Q$, say $\lip(h)+\sup_{\omega\in\Omega}\lip(f_\omega) \leq L$. Let $E_\omega:=\{\t\in Q: f_{\omega}(\t)=h(\t)\}$. If there exists $c>0$ such that 
    \[
        |f_\omega(\t)-f_\eta(\t)| \geq c\rho_*(\omega,\eta), \quad \forall \t\in Q,
    \]
    then $\dist(E_\omega,E_\eta)\gtrsim_{c,L}\rho_*(\omega,\eta)$.
\end{lemma}
\begin{proof}
    Take $\t\in E_\omega$ and $\t'\in E_\eta$. Then 
    \begin{align*}
        c\rho_*(\omega,\eta) \leq |f_\omega(\t)-f_\eta(\t)| &= |h(\t)-f_\eta(\t)| \\
        &\leq |h(\t)-h(\t')|+|f_\eta(\t')-f_\eta(\t)| \leq L|\t-\t'|.
    \end{align*}
    Taking the infimum over $E_\omega$ and $E_\eta$ proves the claim.
\end{proof}


For $\bi,\bj\in\Sigma$ and $k\geq 1$, define 
\[
    U_k(\bi,\bj) := \{\t\in U: \exists \bj'\in[\bj|_k] \text{ such that } \Pi^{\t}(\bj')=\Pi^{\t}(\bi)\}.
\]
So $U_k(\bi,\bj)$ consists of parameters for which the point coded by $\bi$ also has a coding whose first $k$ symbols agree with $\bj$.

\begin{proposition}\label{prop:dimensionstate}
    Assume that $\{\Phi^{\t}\}_{\t\in U}$ satisfies the transversality condition. Let $\bi,\bj\in\Sigma$ satisfy $i_1\neq j_1$. If $\Pi^{\t_0}(\bi)=\Pi^{\t_0}(\bj)$ for some $\t_0\in U$, then for every neighbourhood $V\subset U$ of $\t_0$ and every $\varepsilon>0$, 
    \[
        \dimh (U_k(\bi,\bj)\cap V) \geq d-1+\dimh K^{\t_0}-\varepsilon
    \]
    for all large $k$.
\end{proposition}
\begin{proof}
    Fix $\varepsilon>0$. By Lemma~\ref{lem:sscsubsystem}, there exists a finite $\Gamma\subset\Lambda^*$ such that $\Psi^{\t_0}:=\{\varphi_{\i}^{\t_0}:\i\in\Gamma\}$ satisfies the SSC and has similarity dimension $\tilde{s}(\t_0)\geq \dimh K^{\t_0}-\varepsilon/4$. The coefficients $\lambda_{\i}(\t)$ and $d_{\i}(\t)$ depend continuously on $\t$ for every $\i\in\Gamma$. Hence we may take a small compact neighbourhood $Q_0\subset V$ of $\t_0$ such that the restricted IFS $\Psi^{\t}:=\{\varphi^{\t}_{\i}: \i\in\Gamma\}$ still satisfies the SSC for every $\t\in Q_0$.  
    
    For $\i\in\Gamma$, write $\underline{r}_{\i}:=\min_{\t\in Q_0} |\lambda_{\i}(\t)|$ and let $s_*$ be the unique solution of $\sum_{\i\in\Gamma} \underline{r}_{\i}^{s_*}=1$. Shrinking $Q_0$ if necessary, we may assume that $s_*>\dimh K^{\t_0}-\varepsilon/2$. Note that we can regard $\Omega:=\Gamma^{\mathbb{N}}$ as the symbolic space of the subsystem with the metric 
    \[
        \rho_*(\omega,\eta) = \underline{r}_{\omega\wedge\eta}, \quad \forall\omega,\eta\in\Omega,
    \]
    where $\underline{r}_{\i_1\cdots \i_n}:=\underline{r}_{\i_1}\cdots\underline{r}_{\i_n}$ for $\i_1,\ldots,\i_n\in\Gamma$. The Bernoulli measure with weights $(\underline{r}_{\i}^{s_*})_{\i\in\Gamma}$ is an $s_*$-Frostman measure on $(\Omega,\rho_*)$.
    Also, since the SSC for $\Psi^{\t}$ holds uniformly on $Q_0$, there exists a constant $C_1>0$ such that
    \begin{equation}\label{eq:thesscbd}
    |\Pi^{\t}(\omega)-\Pi^{\t}(\eta)| \geq C_1 \rho_*(\omega,\eta), \quad \forall \omega,\eta\in\Omega, \forall \t\in Q_0.
    \end{equation}

    Set $F(\t):=\Pi^{\t}(\bj)-\Pi^{\t}(\bi)$. Then $F(\t_0)=0$. Since $i_1\neq j_1$, the transversality condition gives $|\nabla F(\t_0)|>0$. Choose a unit vector $v$ with $\partial_v F(\t_0)>0$. By shrinking around $\t_0$ inside $Q_0$, we may choose $a>0$ and a closed ball $B\subset v^\perp$, centered at the origin, such that 
    \[
        Q := \{\t_0+\delta v+z: |\delta|\leq a, z\in B\} \subset Q_0
    \]
    and  
    \[
        \partial_v F(\t) \geq c_0>0, \quad \forall \t\in Q.
    \]
    Picking $a$ and $B$ small enough, we may also arrange that 
    \[
        F(\t_0-av+z)<0 < F(\t_0+av+z), \quad \forall z\in B.
    \]
    For $\omega\in\Omega$, define $F_{k,\omega}(\t):=\Pi^{\t}(\bj|_k\omega)-\Pi^{\t}(\bi)$. By Lemma~\ref{lem:uniformc1}, $F_{k,\omega}\to F$ in $C^1(Q)$ uniformly over $\omega\in\Omega$. So we have for all large $k$ that 
    \begin{equation}\label{eq:parvoffko}
        \partial_v F_{k,\omega}(\t) \geq \frac{c_0}{2}, \quad \forall \t\in Q, \omega\in\Omega
    \end{equation}
    and 
    \[
        F_{k,\omega}(\t_0-av+z)<0 < F_{k,\omega}(\t_0+av+z), \quad \forall z\in B, \omega\in\Omega.
    \]
    Fix such a $k$. 
    
    For any $\omega\in\Omega$ and $z\in B$, the function $\delta \mapsto F_{k,\omega}(\t_0+\delta v+z)$ is strictly increasing on $[-a,a]$ and takes opposite signs at the two endpoints. Hence it has a unique zero $h_{\omega}(z)\in(-a,a)$. Letting 
    \[
        E_\omega := \{\t\in Q: F_{k,\omega}(\t)=0\},
    \]
    we have 
    \begin{equation}\label{eq:formofeomega}
        E_{\omega} = \{\t_0+h_{\omega}(z)v+z: z\in B\}
    \end{equation}
    By~\eqref{eq:parvoffko} and the implicit function theorem, $h_{\omega}$ is $C^1$. Thus $E_\omega$ is a $C^1$ graph over $B$. For convenience, write 
    \[
        \phi_\omega(z) := \t_0+h_\omega(z)v+z, \quad z\in B.
    \]
    So $F_{k,\omega}(\phi_\omega(z))=0$ and $E_\omega=\phi_\omega(B)$ (recall~\eqref{eq:formofeomega}). 

    For any unit vector $w\in v^\perp$, $\partial_w\phi_\omega(z)=(\partial_w h_\omega(z))v+w$. Differentiating $F_{k,\omega}(\phi_\omega(z))=0$ in $w$, it follows from the chain rule that 
    \begin{align*}
        0 &= \partial_w\bigl(F_{k,\omega}\circ\phi_\omega\bigr)(z) \\
        &= DF_{k,\omega}|_{\phi_\omega(z)}(\partial_w\phi_\omega(z)) \\
        &= DF_{k,\omega}|_{\phi_\omega(z)}((\partial_w h_\omega(z))v+w) \\
        &= (\partial_w h_\omega(z))DF_{k,\omega}|_{\phi_\omega(z)}(v)+DF_{k,\omega}|_{\phi_\omega(z)}(w) \\
        &= (\partial_w h_\omega(z))\partial_v F_{k,\omega}(\phi_\omega(z))+\partial_w F_{k,\omega}(\phi_\omega(z)).
    \end{align*}
    The uniform $C^1$-convergence above yields $\sup_{\omega\in\Omega,\t\in Q} |\nabla F_{k,\omega}(\t)|\lesssim 1$. Together with~\eqref{eq:parvoffko}, we see that 
    \[
        |\partial_wh_\omega(z)| = \frac{|\partial_wF_{k,\omega}(\phi_\omega(z))|}{|\partial_vF_{k,\omega}(\phi_\omega(z))|} \lesssim \frac{1}{c_0/2} \lesssim 1
    \]
    and therefore $|\nabla h_\omega|\lesssim 1$ uniformly in $\omega$ and $z$. In particular, the graphing functions $h_\omega$ have uniformly bounded Lipschitz constants. Moreover, the map $(\omega,z)\mapsto h_\omega(z)$ is continuous. The differentiation formula above then shows that $(\omega,z)\mapsto \nabla h_\omega(z)$ is continuous as well. 

    We verify the hypotheses in Lemma~\ref{lem:hypersurface-family}. Since $(\omega,z)\mapsto \phi_\omega(z)$ is continuous, the set 
    \[
        \bigcup_{\omega\in\Omega} E_\omega = \{\phi_\omega(z): (\omega,z)\in\Omega\times B\}
    \]
    is compact. 
    For every Borel $A\subset\R^d$, the area formula implies that 
    \[
        \h^{d-1}(E_\omega\cap A) = \int_B \mathbbm{1}_{A}(\phi_\omega(z)) \sqrt{1+|\nabla h_\omega(z)|^2} \rd\h^{d-1}(z).
    \]
    The integrand is jointly Borel in $(\omega,z)$ and hence the map $\omega\mapsto\h^{d-1}(E_\omega\cap A)$ is Borel. Since $|\nabla h_\omega|\lesssim 1$ uniformly in $\omega$ and $B$ is fixed, $\h^{d-1}(E_\omega)\approx \h^{d-1}(B)\approx 1$ uniformly in $\omega$. Also, the set $\{z\in B: \phi_\omega(z)\in B(\t,r)\}$ is contained in a $(d-1)$-dimensional ball of radius $r$ in $v^\perp$. So its $(d-1)$-dimensional Hausdorff measure is $\lesssim r^{d-1}$. Applying the area formula once more gives 
    \[
        \h^{d-1}(E_\omega\cap B(\t,r))\lesssim r^{d-1}
    \]
    uniformly in $\omega$, $\t$ and $r$.
    
    It remains to verify separation between these hypersurfaces. For $\omega,\eta\in\Omega$ and $\t\in Q$, 
    \[
        |\Pi^{\t}(\bj|_k\omega)-\Pi^{\t}(\bj|_k\eta)| = |\lambda_{\bj|_k}(\t)||\Pi^{\t}(\omega)-\Pi^{\t}(\eta)| \gtrsim_k \rho_*(\omega, \eta),
    \]
    where we used~\eqref{eq:thesscbd} and the fact that $\min_{\t\in Q} |\lambda_{\bj|_k}(\t)|>0$ for fixed $k$. Set $f_\omega(\t):=\Pi^{\t}(\bj|_k\omega)$ and $h(\t):=\Pi^{\t}(\bi)$, Lemma~\ref{lem:uniformc1} gives a uniform Lipschitz bound on $Q$ for $h$ and the family $\{f_\omega\}_{\omega}$. So by Lemma~\ref{lem:zeroseparation}, $\dist(E_\omega,E_\eta)\gtrsim_k \rho_*(\omega,\eta)$. 
    All the hypotheses of Lemma~\ref{lem:hypersurface-family} are therefore satisfied. Thus 
    \[
        \dimh \Big( \bigcup_{\omega\in\Omega} E_\omega \Big) \geq d-1+s_*.
    \]
    For every $\omega\in\Omega$ we have $\bj|_k\omega\in[\bj|_k]$ and thus $\bigcup_{\omega\in\Omega} E_\omega\subset U_{k}(\bi,\bj)\cap Q \subset U_{k}(\bi,\bj)\cap V$. Since $s_*> \dimh K^{\t_0}-\varepsilon/2$, the proof is completed.
\end{proof}

Now we are ready to prove Theorem~\ref{thm:denseinffibre}.

\begin{proof}[Proof of Theorem~\ref{thm:denseinffibre}]
    Fix any $\t_0\in\widetilde{U}$. Since $\Phi^{\t_0}$ fails the SSC, there are $\bi,\bj\in\Sigma$ such that $i_1\neq j_1$ and $\Pi^{\t_0}(\bi)=\Pi^{\t_0}(\bj)$. For $p\geq 1$, write $\bi|_p^\infty:=\bi|_p\bi|_p\cdots$ and define 
    \[
        F(\t) := \Pi^{\t}(\bi)-\Pi^{\t}(\bj) \quad\text{and}\quad F_p(\t) := \Pi^{\t}(\bi|_p^\infty)-\Pi^{\t}(\bj).
    \]
    By Lemma~\ref{lem:uniformc1}, $F_p\to F$ in $C^1$ on a compact neighbourhood of $\t_0$. Since $F(\t_0)=0$ and $i_1\neq j_1$, transversality gives $\nabla F(\t_0)\neq 0$. 
    Similarly as in the proof of Proposition~\ref{prop:dimensionstate}, for all large $p$ we can find $\t_p\to \t_0$ such that $\Pi^{\t_p}(\bi|_p^\infty)=\Pi^{\t_p}(\bj)$. This already proves density. Indeed, writing $x_p:=\Pi^{\t_p}(\bi|_p^\infty)$, we have seen in Example~\ref{exa:periodictype} that $x_p$ has infinitely many codings. Since $\t_p\to \t_0$, the set $T$ is dense in $\widetilde{U}$.

    To prove the local dimension estimate, let $V$ be a neighbourhood of $\t_0$ with $V\subset U$ and fix $\varepsilon>0$. Then for all large $p$, $\t_p\in V$ and by the lower semicontinuity from Lemma~\ref{lem:sscsubsystem}, 
    \[
        \dimh K^{\t_p}>\dimh K^{\t_0}-\varepsilon.
    \]
    Choose a neighbourhood $W$ of $\t_p$ with $\overline{W}\subset V$. Applying Proposition~\ref{prop:dimensionstate} to $\bi|_p^\infty$ and $\bj$, we obtain for all large $k$ that
    \[
        \dimh (U_k(\bi|_p^\infty,\bj)\cap W) \geq d-1+\dimh K^{\t_p} - \varepsilon.
    \]
    Every parameter $\t\in U_{k}(\bi|_p^\infty,\bj)$ belongs to $T$. It follows that 
    \[
        \dimh (T\cap V) \geq d-1+\dimh K^{\t_0}-2\varepsilon.
    \]
    Letting $\varepsilon \to 0$, we have $\dimh (T\cap V) \geq d-1+\dimh K^{\t_0}$. The bound~\eqref{eq:dimofdense2} follows immediately.
\end{proof}




\subsection{Baire-category sharpness of Theorem~\ref{thm:denseinffibre}}
\label{subsec:baireexa}

The construction is obtained by adding one more parameter to the Kamalutdinov--Tetenov family as follows.

Fix $0<\varepsilon<1/100$ and $p,r\in(0,1/50)$. For $(q,\tau)\in U:=(0,\varepsilon)\times(-\varepsilon,\varepsilon)$, consider the IFS $\Phi^{q,\tau}=\{\varphi^{q,\tau}_i\}_{i=1}^6$, where 
\begin{equation*}
    \begin{gathered}
        \varphi^{q,\tau}_1(x) = px, \, \varphi^{q,\tau}_2(x)=rx+\tfrac{1}{5}, \, \varphi^{q,\tau}_3(x) = -qx+\tfrac{8}{15}, \\
        \varphi^{q,\tau}_4(x) = rx+\tfrac{8}{15}-r+\tau, \, \varphi^{q,\tau}_5(x) = -rx+\tfrac{4}{5}, \, \varphi^{q,\tau}_6(x) = rx+1-r,
    \end{gathered}
\end{equation*}
The additional translation parameter $\tau$ appears only in the fourth map. Let $K^{q,\tau}$ and $\Pi^{q,\tau}$ be the attractor and coding map, respectively. Here the alphabet is $\Lambda=\{1,\ldots,6\}$ and the symbolic space is $\Sigma=\Lambda^{\mathbb{N}}$. For each finite word $\i\in\Lambda^*$, we write $K^{q,\tau}_{\i}:=\varphi^{q,\tau}_{\i}(K^{q,\tau})$.

Every map $\varphi_{i}^{q,\tau}$ sends the unit interval $[0,1]$ into $[0,1]$. So $K^{q,\tau}\subset[0,1]$. Moreover, $0,1\in K^{q,\tau}$. So $K^{q,\tau}$ is nontrivial and perfect. Our choices of $\varepsilon,p,r$ imply that 
\begin{equation}\label{eq:distofotherpiece}
    \dist(\varphi^{q,\tau}_i([0,1]), \varphi^{q,\tau}_j([0,1])) \geq \frac{1}{100}, \quad \forall (q,\tau)\in U
\end{equation}
whenever $i\neq j$ and $\{i,j\}\neq\{3,4\}$. In other words, the only possible intersection between distinct level-$1$ pieces is $K_3^{q,\tau}\cap K^{q,\tau}_4$.

As in Theorem~\ref{thm:denseinffibre}, let 
\[
    \widetilde{U} := \{(q,\tau)\in U: \Phi^{q,\tau} \text{ fails the SSC} \}
\]
and let
\[
    T := \{(q,\tau)\in U: \exists x\in K^{q,\tau} \text{ with infinitely many codings} \}
\]
Since the SSC is an open condition, $\widetilde{U}$ is relatively closed in $U$. Define the vertical sections 
\[
    \widetilde{U}_q:=\{\tau\in (-\varepsilon,\varepsilon): (q,\tau)\in\widetilde{U}\}, \quad q\in(0,\varepsilon).
\]

Fix $q$. Recall that the contraction ratios are independent of $\tau$ and have absolute value less than $1/50$. It follows that for every $\omega\in\Sigma$,
\begin{equation*}
    \bigl| \partial_\tau\Pi^{q,\tau}(\omega) \bigr| \leq \sum_{k=0}^\infty \max\{p,q,r\}^k \leq \sum_{k=0}^{\infty} 50^{-k} = \frac{50}{49}.
\end{equation*}
For $\omega,\eta\in\Sigma$, set
\[
    F_{\omega,\eta}^q(\tau) := \Pi^{q,\tau}(4\eta)-\Pi^{q,\tau}(3\omega) = \tau-r+r\Pi^{q,\tau}(\eta)+q\Pi^{q,\tau}(\omega).
\]
Then $F_{\omega,\eta}^q$ is affine and
\begin{equation}\label{eq:tau-transversality}
    \frac{\mathrm d}{\mathrm d\tau}F_{\omega,\eta}^q(\tau) \geq 1-\frac{50}{49}(q+r) > 1-\frac{50}{49}\cdot\frac{2}{50} > \frac{1}{2}.
\end{equation}

\begin{lemma}\label{lem:endpoint-2codes}
    Let $q\in(0,\varepsilon)$ and let $\tau_0\in\widetilde{U}_q$. If for some $\delta>0$ either $(\tau_0-\delta,\tau_0)\subset (-\varepsilon,\varepsilon)\setminus\widetilde{U}_q$ or $(\tau_0,\tau_0+\delta)\subset (-\varepsilon,\varepsilon)\setminus\widetilde{U}_q$, then every point of $K^{q,\tau_0}$ has at most two codings.
\end{lemma}
\begin{proof}
    Assume that $(\tau_0-\delta,\tau_0)\subset (-\varepsilon,\varepsilon)\setminus\widetilde{U}_q$ (the other case can be proved analogously). 
    Let $z\in K^{q,\tau_0}_3\cap K^{q,\tau_0}_4$, say $z=\varphi_3^{q,\tau_0}(x)=\varphi_4^{q,\tau_0}(y)$. Pick codings $\omega,\eta\in\Sigma$ of $x$ and $y$, respectively. We claim that $x$ is a right endpoint in $K^{q,\tau_0}$. Otherwise, there exists a sequence $x_n\in K^{q,\tau_0}$ with $x_n\downarrow x$. Let $\omega_n$ be a coding of $x_n$. Then
    \begin{align*}
        F_{\omega_n,\eta}^q(\tau_0) &= \Pi^{q,\tau_0}(4\eta) -\Pi^{q,\tau_0}(3\omega_n) \\
        &= \varphi_4^{q,\tau_0}(y) - \varphi_3^{q,\tau_0}(x_n) \\
        &=\varphi_3^{q,\tau_0}(x) - \varphi_3^{q,\tau_0}(x_n) \\
        &=q(x_n-x)>0.
    \end{align*}
    Since $F_{\omega_n,\eta}^q$ is affine and $\frac{\mathrm d}{\mathrm d\tau}F_{\omega_n,\eta}^q(\tau)>1/2$ (recall~\eqref{eq:tau-transversality}), for all large $n$ there exists a unique $\tau_n<\tau_0$ such that $F_{\omega_n,\eta}^q(\tau_n)=0$. Moreover, 
    \[
        0<\tau_0-\tau_n \leq 2q(x_n-x) \to 0.
    \]
    So $\tau_n\in(\tau_0-\delta,\tau_0)$ for all large $n$ while $\Pi^{q,\tau_n}(4\eta)=\Pi^{q,\tau_n}(3\omega_n)$, that is, $\tau_n\in\widetilde{U}_q$. This is a contradiction.
    
    Similarly, $y$ is a right endpoint of $K^{q,\tau_0}$. Since $K^{q,\tau_0}$ is perfect, $x$ and $y$ are both left accumulation points of $K^{q,\tau_0}$. The map $\varphi^{q,\tau_0}_3$ reverses orientation, whereas $\varphi^{q,\tau_0}_4$ preserves it. So $z$ is two-sided accumulated. Since $z$ was arbitrary, this holds for every point in the critical set $K^{q,\tau_0}_3\cap K^{q,\tau_0}_4$.

    Combining the claim with~\eqref{eq:branchingorbit}, we see that every point in $K^{q,\tau_0}$ with at least two codings is two-sided accumulated. This in turn implies that every $z\in K^{q,\tau_0}_3\cap K^{q,\tau_0}_4$ has exactly two codings, because $x$ and $y$ are uniquely coded (as one-sided endpoints). The desired conclusion then follows immediately from Lemma~\ref{lem:Eisenough}(1).
\end{proof}

\begin{lemma}\label{lem:uq-close-zerolength}
    For every $q\in(0,\varepsilon)$, the set $\widetilde{U}_q$ is closed in $(-\varepsilon,\varepsilon)$ and $\dimh\widetilde{U}_q < 1$.
\end{lemma}
\begin{proof}
    Relative closedness follows from that of $\widetilde U$ in $U$. 
    Fix $q$. For $n\geq 1$ and $\i,\j\in\Lambda^n$, define 
    \[
        E_{\i,\j} := \{\tau\in(-\varepsilon,\varepsilon): K^{q,\tau}_{3\i}\cap K^{q,\tau}_{4\j}\neq\varnothing\}.
    \]
    If $\tau\in E_{\i,\j}$, then from $K^{q,\tau}\subset[0,1]$ and the contraction bound we see that  
    \[
        |\varphi^{q,\tau}_{4\j}(0)-\varphi^{q,\tau}_{3\i}(0)| \leq \diam(K^{q,\tau}_{3\i})+\diam(K^{q,\tau}_{4\j}) \leq (q+r)50^{-n}.
    \]
    By the transversality condition~\eqref{eq:tau-transversality}, it is standard to show that $E_{\i,\j}$ is contained in an interval of length at most $4(q+r)50^{-n}$. Note that $\widetilde{U}_q\subset\bigcup_{\i,\j\in\Lambda^n}E_{\i,\j}$. There are at most $36^n$ such sets, so for every $s>\frac{\log 36}{\log 50}$, the sum 
    \[
        \sum_{\i,\j\in\Lambda^n} (\diam E_{\i,\j})^s \lesssim 36^n\cdot (q+r)^s50^{-ns} \lesssim (36\cdot 50^{-s})^n \to 0.
    \]
    Therefore, $\dimh\widetilde{U}_q\leq\frac{\log 36}{\log 50}<1$.
\end{proof}

For convenience, write 
\[
    T_2:=\{(q,\tau)\in\widetilde{U}: \text{every point of $K^{q,\tau}$ has $\leq 2$ codings}\}.
\]
By Lemma~\ref{lem:uq-close-zerolength}, each $\widetilde{U}_q$ is closed and nowhere dense in $(-\varepsilon,\varepsilon)$. So the endpoints of the connected components of $(-\varepsilon,\varepsilon)\setminus\widetilde{U}_q$ that lie in $(-\varepsilon,\varepsilon)$ are dense in $\widetilde{U}_q$. Lemma~\ref{lem:endpoint-2codes} shows that every such endpoint $\tau_*$ satisfies $(q,\tau_*)\in T_2$. Approximating within each vertical section, we conclude that $T_2$ is dense in $\widetilde{U}$.

For pairwise distinct words $u,v,w\in\Lambda^n$, define 
\[
    B_{u,v,w} := \{(q,\tau)\in\widetilde{U}: K^{q,\tau}_u\cap K^{q,\tau}_v\cap K^{q,\tau}_w\neq\varnothing\}.
\]
By continuity of the cylinder sets in the Hausdorff metric, each $B_{u,v,w}$ is relatively closed in $\widetilde{U}$. Moreover, it is disjoint from the dense set $T_2$. So it is nowhere dense in $\widetilde{U}$. Then 
\[
    T \subset \widetilde{U}\setminus T_2\subset \bigcup_{n=1}^\infty\bigcup_{u,v,w\in\Lambda^n \text{ pairwise distinct}} B_{u,v,w}
\]
is of the first category in $\widetilde{U}$.

\subsection{Dimension of the critical set}

Finally, we address the upper bound for the dimension of the overlap set. The proof is based on the standard combination of a transversality estimate and a stopping-time covering argument.

\begin{proof}[Proof of Theorem~\ref{thm:transfam_drop}]
    Let $V\subset U$ be open with $\overline{V}\subset U$. Denote by $D_V$ the supremum of diameters of $K^{\t}$, $\t\in V$. For $i\in\Lambda$, set $b_i:=\sup_{\t\in V}|r_i(\t)|$, and for $\i=i_1\cdots i_n$ define $b_{\i}:=b_{i_1}\cdots b_{i_n}$. Let $s_V$ be the unique solution of $\sum_{i\in\Lambda} b_i^{s_V}=1$. For $0<\delta<1$, set 
    \[
        \mathcal{W}_\delta := \{u\in\Lambda^*: b_{u}\leq \delta<b_{u^-}\}.
    \]
    Note that $b_{u}\approx_V\delta$ for $u\in\mathcal{W}_\delta$ and $\sum_{u\in\mathcal{W}_\delta} b_u^{s_V}=1$. So 
    \begin{equation}\label{eq:tranddp_1}
        \#\mathcal{W}_\delta \lesssim_V \delta^{-s_V}.
    \end{equation}

    Let $i,j\in\Lambda$ be distinct. For each pair $u,v\in\mathcal{W}_\delta$, set 
    \[
        E_{u,v} := \{\t\in V: \varphi^{\t}_{iu}(K^{\t})\cap\varphi^{\t}_{jv}(K^{\t})\neq\varnothing\}.
    \]
    If $\t\in E_{u,v}$, then the two cylinder images intersect, so 
    \[
        |\Pi^{\t}(iu\xi)-\Pi^{\t}(jv\xi)| \leq 2D_V\delta, \quad \forall\xi\in\Sigma.
    \]
    Since $i\neq j$, we have by the transversality lemma quoted above that 
    \begin{equation}\label{eq:tranddp_2}
        \mathcal{L}^d(E_{u,v}) \leq C_V\cdot 2D_V\delta \lesssim_V \delta.
    \end{equation}

    For each $\t\in V$, the overlap set $Q_{i,j}(\t) := \varphi^{\t}_i(K^{\t})\cap \varphi^{\t}_{j}(K^{\t})$ is covered by the family
    \[
        \big\{ \varphi^{\t}_{iu}(K^{\t})\cap\varphi^{\t}_{jv}(K^{\t}): u,v\in\mathcal{W}_\delta, \t\in E_{u,v} \big\}
    \]
    and each member of the cover has diameter at most $D_V\delta$. Hence 
    \[
        N_{D_V\delta}(Q_{i,j}(\t)) \leq \sum_{u,v\in\mathcal{W}_\delta} \mathbbm{1}_{E_{u,v}}(\t).
    \]
    Integrating over $V$ and using~\eqref{eq:tranddp_1}--\eqref{eq:tranddp_2}, we obtain the averaged covering estimate 
    \begin{equation}\label{eq:tranddp_3}
        \int_V N_{D_V\delta}(Q_{i,j}(\t)) \rd\t \leq \int_V \sum_{u,v\in\mathcal{W}_\delta} \mathbbm{1}_{E_{u,v}}(\t) \rd\t = \sum_{u,v\in\mathcal{W}_\delta}\mathcal{L}^d(E_{u,v}) \lesssim_V \delta^{1-2s_V}.
    \end{equation}
    
    Let $\alpha:=\max\{2s_V-1,0\}$. For every $\varepsilon>0$, Markov's inequality and~\eqref{eq:tranddp_3} give (take $\delta=2^{-n}$)
    \[
        \mathcal{L}^d(\{\t\in V: N_{D_V/2^n}(Q_{i,j}(\t))>2^{n(\alpha+\varepsilon)}\}) \lesssim_V 2^{-n(1-2s_V+\alpha+\varepsilon)}.
    \]
    The right hand side is summable in $n$: the exponent in the parentheses equals $\varepsilon$ when $s_V\geq 1/2$, and it is $1-2s_V+\varepsilon>0$ when $s_V<1/2$. By Borel-Cantelli, for almost every $\t\in V$ and all large $n$, 
    \[
        N_{D_V/2^n}(Q_{i,j}(\t)) \leq 2^{n(\alpha+\varepsilon)}.
    \]
    Letting $n\to\infty$, we have $\dimub Q_{i,j}(\t)\leq\alpha+\varepsilon$. Since $\varepsilon$ was arbitrary, 
    \[
        \dimub Q_{i,j}(\t) \leq \max\{2s_V-1, 0\}
    \]
    for almost every $\t\in V$.

    To pass from $s_V$ to the pointwise value $s(\t)$, choose a countable basis $\mathcal G$ of relatively compact open subsets of $U$ and intersect the full-measure sets obtained above for all $V\in\mathcal G$ and all pairs $i\neq j$. For every $\t$ in the resulting full-measure set, choose $V_k\in\mathcal{G}$ shrinking to $\t$ with $\t\in V_k$. Then $s_{V_k}\to s(\t)$. It follows that
    \[
        \dimub Q_{i,j}(\t) \leq \max\{0,2s(\t)-1\} \quad\text{ for a.e. } \t\in U.
    \]
    Summing over the finitely many pairs $i\neq j$ completes the proof.
\end{proof}

\begin{example}
    Consider the $\{0,1,3\}$-family $\Phi^\lambda = \{\lambda x, \lambda x+1, \lambda x+3\}$, where $\lambda\in (\frac{1}{4},\frac{1}{3})=:J$, and let $K^\lambda$ denote the attractor. It satisfies the transversality condition on $J$; see~\cite[Chapter 6]{BSS23}. Moreover, the ESC holds for Lebesgue almost every $\lambda\in J$, whereas the OSC fails for every $\lambda\in J$. For each such parameter, the attractor and the associated natural self-similar measure have the expected dimension $s(\lambda)=\frac{\log 3}{-\log\lambda}$. By Theorem~\ref{thm:transfam_drop}, 
    \[
        \dimh \mathcal{C}^\lambda \leq \dimub\mathcal{C}^\lambda \leq 2\frac{\log 3}{-\log\lambda} - 1 < \frac{\log 3}{-\log\lambda} = s(\lambda)
    \]
    for almost all $\lambda\in J$. It follows from~\eqref{eq:branchingorbit} that the natural self-similar measure gives full mass to the set of uniquely coded points.
\end{example}

\section{Open questions}

We conclude with three open questions. Let $K\subset\R$ be the self-similar set generated by an IFS $\Phi=\{\varphi_i\}_{i\in\Lambda}$ with contraction ratios $(r_i)_{i\in\Lambda}$ and similarity dimension $s$. Let $\mu$ be the natural Bernoulli measure on the symbolic space with weights $(r_i^s)_{i\in\Lambda}$ and let $\nu=\Pi\mu$.

\begin{question}
    If every point of $K$ has finitely many codings, must the
    numbers of codings be uniformly bounded?
\end{question}

Theorem~\ref{thm:fipimplybd} answers the question affirmatively under the FIP. So it remains to consider the case when $\mathcal{C}$ is infinite. A direct compactness argument does not suffice. Indeed, suppose that $x_n\in K$ and $\#\Pi^{-1}(x_n)\to\infty$. Passing to a subsequence, we may assume that $x_n\to x$. By continuity of $\Pi$, every convergent sequence of codings of the points $x_n$ has its limit in $\Pi^{-1}(x)$. However, distinct sequences of codings may converge to the same address if they share initial blocks of arbitrarily large length. 

\begin{question}\label{que:unique-coding}
    If $\dimh\nu=s$, does $\nu$-almost every point of $K$ have a unique coding?
\end{question}

Equivalently, is there a Borel set $\widetilde{B}\subset\Sigma$ with $\mu(\widetilde B)=0$ such that $\Pi^{-1}(\Pi(\bi))=\{\bi\}$ for every $\bi\in\Sigma\setminus\widetilde{B}$? 
As explained after Proposition~\ref{prop:fibrebound}, the no-dimension-drop condition $\dimh\nu=s$ is equivalent to the existence of a Borel set $B\subset\Sigma$ with $\mu(B)=0$ such that $\Pi$ is injective on $\Sigma\setminus B$. This gives uniqueness after removing an exceptional set of codings. But Question~\ref{que:unique-coding} concerns the full fibres, including any additional codings in that exceptional set. By the ergodicity of $\mu$, if the answer is negative then we have no dimension drop yet $\nu$-almost every point has infinitely many codings. The analogous implication for arbitrary positive Bernoulli weights $(p_i)_{i\in\Lambda}$ is false when $s$ is replaced by the similarity dimension of the measure. An example due to De-Jun Feng appears in~\cite[Example 11.5.5]{BSS23}.



\begin{question}\label{que:013-codings}
    For $\lambda\in(1/4,1/3)$, let $K^\lambda$ be the attractor of the $\{0,1,3\}$-family
    \[
        \Phi^\lambda=\{\lambda x,\lambda x+1,\lambda x+3\}.
    \]
    Does the set of parameters $\lambda$ for which $K^\lambda$ contains a point with infinitely many codings have positive Lebesgue measure?
\end{question}

It is known that the OSC fails for every $\lambda\in(1/4, 1/3)$; see~\cite[Proposition 7.7.4]{BSS23}. By Theorem~\ref{thm:denseinffibre}, the parameter set in Question~\ref{que:013-codings} is dense in $(1/4,1/3)$ and has Hausdorff dimension one (since $\sup_{\lambda\in(1/4,1/3)}\dimh K^\lambda=1$). These conclusions do not determine its Lebesgue measure.







\bigskip
\noindent{\bf Acknowledgements.} 
The authors are grateful to De-Jun Feng for his helpful comments.  
KS is supported by the National Research, Development and Innovation Oﬃce--NKFIH, Hungary, Projects K152858 and K144059. JX is partially supported by NSFC grant No. 12501113 and the Fundamental Research Funds for the Central Universities No. NS2025025. The authors are also both supported by the Foreign Expert Program No. H20250925. 

\medskip
\noindent{\bf Declaration on AI assistance.}
Generative AI tools were used during the preparation of this manuscript to assist with literature searches, proof checking, and to improve the clarity and organization of the exposition. All mathematical results, arguments, and proofs were developed and verified by the authors.

\small
\bibliographystyle{amsplain}

\end{document}